\documentclass[12pt,leqno]{amsart}
\usepackage{amsmath, amssymb, amscd, amsfonts, xypic,   stmaryrd, turnstile, mathrsfs, eucal, color}

\definecolor{dullmagenta}{rgb}{0.4,0,0.4}

\definecolor{darkblue}{rgb}{0,0,0.4}

\usepackage[colorlinks=true, pdfstartview=FitV, linkcolor=red, citecolor=blue, urlcolor=darkblue]
{hyperref}

\newtheorem{theorem}{Theorem}[section] 
\newtheorem{lemma}[theorem]{Lemma}
\newtheorem{proposition}[theorem]{Proposition}
\newtheorem{corollary}[theorem]{Corollary}

\theoremstyle{definition}
\newtheorem{definition}[theorem]{Definition}
\newtheorem{example}[theorem]{Example}
\newtheorem{remark}[theorem]{Remark}

\begin{document}

\title[Ring extensions of length two]{Ring extensions of length two }

\author[G. Picavet and M. Picavet]{Gabriel Picavet and Martine Picavet-L'Hermitte}
\address{Math\'ematiques \\ 8 Rue du Forez\\ 63670 Le Cendre   \\ France}

\email{Gabriel.Picavet@math.univ-bpclermont.fr}
\email{picavet.mathu(at)orange.fr}

\begin{abstract} We characterize extensions of commutative rings $R\subset S$ such that $R\subset T$ is minimal for each $R$-subalgebra $T$ of $S$ with $T\neq R,S$. This property is equivalent to $R\subset S$ has length 2. Such extensions are either pointwise minimal or simple. We are able to compute the number of subextensions of $R\subset S$. Besides commutative algebra considerations, our main result is a consequence of the recently introduced by van Hoeij et al. concept of principal subfields of a finite separable field extension.    As a corollary of this paper, we  get that simple extensions  of length 2 have  FIP.

\end{abstract}

\subjclass[2010]{Primary:13B02,13B21, 13B22, 12F10;  Secondary: 13B30}

\keywords  {FIP, FCP extension, minimal extension, length of an extension, integral extension, support of a module, t-closure, algebraic field extension,  separable field extension, principal subfield, pointwise minimal extension}

\maketitle

 \section{Introduction and Notation}
 
 This paper has twin objectives. One of them is to answer a question on length 2 ring extensions, raised when writing our   earlier joint paper with P.-J. Cahen   on pointwise minimal extensions  \cite{Pic 7}. Indeed co-pointwise minimal extensions have length 2. Are there other extensions of length 2? The other  objective is to study towers of two minimal ring extensions. Dobbs and Shapiro already  considered them 
 without the above length condition \cite{DS}. Our methods are completely different and lead to a characterization which allows us to compute cardinalities of sets of intermediate extensions and then to answer a question addressed by these authors. As a deep consequence, we show that length 2 simple extensions have FIP. The terminology is explained in the next paragraphs  and  in Section 4 for t-closedness. 

We consider the category of commutative and unital rings. Epimorphisms are those of this category. 
Let $R\subseteq S$ be a (ring) extension. Its conductor is denoted by $(R:S)$ and the set of all $R$-subalgebras of $S$ by $[R,S]$. We set $]R,S[: =[R,S]\setminus \{R,S\}$ (with a similar definition for $[R,S[$ or $]R,S]$). Moreover,  $\overline R$ is the integral closure of $R$ in $S$.  Any writing $[R,S]$ supposes that there is an extension $R\subseteq S$. 

The extension $R\subseteq S$ is said to have FIP (or is called an FIP extension) (for the ``finitely many intermediate algebras property") if $[R,S]$ is finite. A {\it chain} of $R$-subalgebras of $S$ is a set of elements of $[R,S]$ that are pairwise comparable with respect to inclusion. We say that an  extension $R\subseteq S$ has FCP (or is called an FCP extension)  (for the ``finite chain property") if each chain in $[R,S]$ is finite.  An  extension $R\subseteq S$  is called an FMC extension   if there exists a finite maximal chain from $R$ to $S$. D. Dobbs and the authors characterized FCP and FIP extensions \cite{DPP2}. Our main tool will be the minimal (ring) extensions, a concept that was introduced by Ferrand-Olivier \cite{FO}. Recall that an extension $R\subset S$ is called {\it minimal} if $[R, S]=\{R,S\}$. The key connection between the above ideas is that if $R\subseteq S$ has FCP, then any maximal (necessarily finite) chain of $R$-subalgebras of $S$, $R=R_0\subset R_1\subset\cdots\subset R_{n-1}\subset R_n=S$, with {\it length} $n <\infty$, results from juxtaposing $n$ minimal extensions $R_i\subset R_{i+1},\ 0\leq i\leq n-1$. For any extension $R\subseteq S$, the {\it length} of $[R,S]$, denoted by $\ell[R,S]$, is the supremum of the lengths of chains of $R$-subalgebras of $S$. It should be noted that if $R\subseteq S$ has FCP, then there {\it does} exist some maximal chain of $R$-subalgebras of $S$ with length $\ell[R,S]$ \cite[Theorem 4.11]{DPP3}.

In an earlier paper \cite{Pic 7}, P-J. Cahen and the authors  characterized pointwise minimal extensions, a concept introduced by P.-J. Cahen, D. Dobbs and T. Lucas  \cite{CDL}.  An extension $R\subset S$ is called a {\it simple (or monogenic) extension} if $S=R[t]$ for some $t\in S$ and is called a {\it pointwise minimal extension} if $R\subset R[t]$ is  minimal  for each $t\in S\setminus R$, whereas it  is called a {\it co-pointwise minimal extension}  if $R[x]\subset S$ is  minimal  for each $x\in S\setminus R$  \cite{Pic 7}. In particular, $R\subset S$ is a co-pointwise minimal extension $\Rightarrow R\subset S$ is a pointwise minimal extension.

In the present work, we study a notion connected to the previous contexts.  We will temporarily  call an extension $R\subset S$   a  {\it  minimal pair}  (resp. a {\it  co-minimal pair}) if  each extension $R\subset T$ (resp. $T\subset S$)  is minimal  for $T\in ]R,S[$.  
We show in Section 2 that such extensions coincide with length 2 extensions and they are either pointwise minimal or simple (Proposition~\ref{2.2}). Dobbs and  Shapiro considered a close situation, namely extensions that are a tower of two minimal extensions. Their main  theorem \cite[Theorem 4.1]{DS} characterizes such extensions having FIP, exhibiting 13 mutually exclusive cases. We recall them in (Theorem ~\ref{6.11}). Their extensions are not necessarily of length 2  and worse:  they may have an infinite length. The present paper is written with a point of view different from Dobbs-Shapiro's. By the way, we answer two questions raised in \cite[Remark 2.11 (a)]{D1} by  Dobbs. As usual, Spec$(R)$ and Max$(R)$ are the set of prime and maximal ideals of a ring $R$, and Rad$(R)$ is the (Jacobson) radical of $R$. We recall that the support of an $R$-module $E$ is $\mathrm{Supp}_R(E):=\{P\in\mathrm{Spec}(R)\mid E_P\neq 0\}$, and $\mathrm{MSupp}_R(E):=\mathrm{Supp}_R(E)\cap\mathrm{Max}(R)$. We say that $R\subset S$ is {\it locally minimal} if $R_M\subset S_M$ is minimal for each $M\in\mathrm{Supp}_R(S/R)$. 
Next notions and results are  involved in our study. Recall  that an extension $R\subseteq S$ is called {\it Pr\"ufer} if $R\subseteq T$ is a flat epimorphism  for each $T\in [R,S]$ \cite{KZ}. An extension   $R\subseteq S$ is classically called a {\it normal} pair if $T\subseteq S$ is integrally closed for each $T\in [R,S]$.  Then  $R\subseteq S$ is Pr\"ufer if and only if it is a normal pair \cite[Theorem 5.2(4)]{KZ}. In \cite{Pic 5}, we observed that an extension is a minimal flat epimorphism  if and only if it is  Pr\"ufer and minimal. In this paper, we also call an extension $R\subseteq S$ {\it quasi-Pr\"ufer } if $\overline R\subseteq S$ is a Pr\"ufer extension \cite{Pic 5}.   An FMC extension is quasi-Pr\"ufer. From now on, we use these terminologies.

A first result is that a ring extension $R\subset S$ of length 2 is quasi-Pr\"ufer and $|\mathrm{Supp}_R(S/R)|\leq 2$ (Proposition~\ref{2.3}). In Section 3, we characterize extensions $R\subset S$ of length 2, such   that $|\mathrm{Supp}_R(S/R)|= 2$. These extensions are simple. In Section 4, we characterize extensions $R\subset S$ of length 2,  such that $|\mathrm{Supp}_R(S/R)|= 1$. The quasi-Pr\"ufer property of these extensions induces a first characterization where the integral closure is involved. For integral extensions, seminormalizations and  t-closures are also involved. 

Examples attest that all cases of our exhaustive classification occur. For  a length 2 t-closed extension $R\subset S$, such that $\mathrm{Supp}_R(S/R)=\{M\}$, we get that $M=(R:S)$. This allows us to reduce our study to the field extension $R/M\subset S/M$. Finite separable field extensions are surprisingly difficult to handle and need the whole Section 5. To get conditions in order that a field extension $k\subset L$ has length 2, we need a tight study of the $k$-subalgebras of $L$. We use the noteworthy notion of principal subfields introduced by van Hoeij, Kl\"uners and Novocin  \cite{HKN}. Here again, several examples show that  various situations may occur. Section 6 collects in the preceding sections, 11 mutually exclusive and comprehensive conditions that  extensions $R \subset S$ do verify to have length 2. For each of them, we give the cardinality of $[R,S]$. Our classification differs from Dobbs-Shapiro's  as the reader may see with the help of a comparative table.

A {\it local} ring is here what is called elsewhere a quasi-local ring.  The characteristic of an integral domain $k$ is denoted by $\mathrm{c}(k)$. If $E$ is an $R$-module, ${\mathrm L}_R(E)$ is its length. If $R\subseteq S$ is a ring extension and $P\in\mathrm{Spec}(R)$, then $S_P$ is both the localization $S_{R\setminus P}$ as a ring and the localization at $P$ of the $R$-module $S$.
Finally, $\subset$ denotes proper inclusion, $|X|$ the cardinality of a set $X$, for a positive integer $n$, we set $\mathbb{N}_n:=\{1,\ldots,n\}$  and $\mathbb{P}$ is the set of all prime numbers. In a ring $R$, for $a,b,c\in R$ such that $c$ divides $a-b$, we write $a\equiv b\ (c)$.

\begin{definition}\label{1.1}  Let $R\subseteq S$ be a ring extension and $M\in\mathrm{Spec}(R)$. We say that $M$ is the {\it crucial  ideal} $\mathcal{C}(R,S)$ of the extension if $\mathrm{Supp}_R(S/R)=\{M\}$. Such an extension is called $M$-{\it crucial}. A crucial ideal needs to be maximal because a support is stable under specialization.
\end{definition}

 \begin{proposition}\label{1.11}  Let $R \subset S$ be an extension, with conductor $C$. The following statements hold:
 
 (1) If $R\subset S$ is $M$-crucial, then $C \subseteq M$.
 
 (2)  If  $R\subset S$ is integral, then $R\subset S$  has a crucial  ideal if and only if $\sqrt C \in \mathrm{Max}(R)$,  and then  $\mathcal C (R,S)= \sqrt C$.
\end{proposition}

\begin{proof}   

(1) If the extension is $M$-crucial, suppose that there is some $x \in C\setminus M$, then it is easily seen that $R_M=S_M$, a contradiction.

We denote by $ \{R_\alpha \mid \alpha \in I\}$   the family of all  finite extensions 
$R \subset R_\alpha $ with $R_\alpha \in [R,S]$   and  conductor $C_\alpha$.

 (2) For  $M \in \mathrm{Spec}(R)$,  observe that $M$ is a crucial ideal of  $R \subset S$    if and only if  $M$ is a crucial ideal of  each $R\subset R_\alpha$. Then it is enough to use the following facts: $\mathrm{Supp}(R_\alpha /R) = \mathrm{V}(C_\alpha)$ and $C = \cap [C_\alpha \mid \alpha \in I]$.  
\end{proof}

\begin{theorem}\label{1.2}  \cite[D\' efinition 1.1]{FO} and \cite[Theorem 3.2]{Pic 7}   A  pointwise minimal extension $A\subset B $ admits  a crucial (maximal) ideal $M$ and is either integral or Pr\"ufer minimal,  these conditions being mutually exclusive. Moreover,  a minimal extension is a simple  algebra. 
\end{theorem}

Three types of minimal integral extensions exist, characterized in the next theorem, (a consequence of  the fundamental lemma of Ferrand-Olivier), so that there are four types of minimal extensions.

\begin{theorem}\label{1.3} \cite [Theorems 2.2 and 2.3]{DPP2} Let $R\subset T$ be an extension and  $M:=(R: T)$. Then $R\subset T$ is minimal and finite if and only if $M\in\mathrm{Max}(R)$ and one of the following three conditions holds:

\noindent (a) {\bf inert case}: $M\in\mathrm{Max}(T)$ and $R/M\to T/M$ is a minimal field extension;

\noindent (b) {\bf decomposed case}: There exist $M_1,M_2\in\mathrm{Max}(T)$ such that $M= M _1\cap M_2$ and the natural maps $R/M\to T/M_1$ and $R/M\to T/M_2$ are both isomorphisms; or, equivalently, there exists $q\in T\setminus R$ such that $T=R[q],\  q^2-q\in M$, and $Mq\subseteq M$.

\noindent (c) {\bf ramified case}: There exists $M'\in\mathrm{Max}(T)$ such that ${M'}^2 \subseteq M\subset M',\  [T/M:R/M]=2$, and the natural map $R/M\to T/M'$ is an isomorphism; or, equivalently, there exists $q\in T\setminus R$ such that $T=R[q],\ q^2\in M$, and $Mq\subseteq M$.
\end{theorem}

We give here two lemmas  used in earlier papers and recall some needed results.  

\begin{lemma} \label{1.12}  Let $R\subset S$ be an  extension and $T,U\in[R,S]$ such that $R\subset T$ is a finite minimal  extension and $R\subset U$ is a Pr\"ufer minimal extension. Then, $\mathcal{C}(R,T)\neq\mathcal{C}(R,U)$, so that $R$ is not a local ring.
\end{lemma}

\begin{proof} Assume that $\mathcal{C}(R,T)=\mathcal{C}(R,U)$ and set $M:=\mathcal{C}(R,T)=(R:T)=\mathcal{C}(R,U)\in\mathrm{Max}(R)$. Then, $MT=M$ and  $MU=U$ because $R\subset U$ is a Pr\"ufer minimal extension \cite[Scholium A (1)]{Pic 5}. It follows that $MUT=UT=MTU=MU=U$, a contradiction.
\end{proof}

\begin{proposition}\label{1.14} \cite[Corollary 3.2]{DPP2} If there exists a maximal chain $R=R_0 \subset\cdots\subset  R_i \subset\cdots \subset R_n=S$ of extensions, where $R_i\subset R_{i+1}$ is minimal, then  $\mathrm{Supp}(S/R)=\{ \mathcal C (R_i, R_{i+1})\cap R\mid i=0,\ldots,n-1\}$.
\end{proposition}

{\begin{lemma} \label{1.15}  Let $R\subset S$ be an FMC extension. If $M\in\mathrm{MSupp}(S/R)$, there exists $T\in[R,S]$ such that $R\subset T$ is minimal with $\mathcal{C}(R,T)=M$. 
\end{lemma}

\begin{proof} Let $\{R_i\}_{i=0}^n$ be a finite maximal chain such that $R_0:=R$ and $R_n:=S$. If $\mathcal{C}(R,R_1)=M$, then, $T=R_1$. So, assume that $M\neq \mathcal{C}(R,R_1)$. Let $k\in\{1,\ldots, n-1\}$ be the least integer $i$ such that $M=\mathcal{C}(R_i,R_{i+1})\cap R$ (Proposition~\ref{1.14}). For each $i<k$, we have $M\neq\mathcal{C}(R_i,R_{i+1})\cap R$, so that $M\in\mathrm{Max}(R)\setminus \mathrm{MSupp}(R_k/R)$. In view of \cite[Lemma 1.10]{Pic 3}, there exists $T\in[R,R_{k+1}]$ such that $R\subset T$ is minimal (of the same type as $R_k\subset R_{k+1}$) with $\mathcal{C}(R,T)=M$. 
\end{proof}

\begin{lemma}\label{1.13} (Crosswise exchange) \cite[Lemma 2.7]{DPP2} Let $R\subset S$ and $S\subset T$ be  minimal  extensions,  $M:= \mathcal{C}(R,S)$, $N:= \mathcal{C}(S,T)$ and $P:=N\cap R$ be such that $P\not\subseteq M$. Then there is   $S' \in [R,T]$ such that $R\subset S'$  is minimal  of the same type as $S\subset T$ and $P= \mathcal{C}(R,S')$; and $S'\subset T$ is  minimal  of the same type as $R\subset S$ and $MS' =\mathcal{C}(S',T)$. Moreover,  $[R,T]=\{R,S,S',T\}$ and $R_Q=S'_Q=T_Q$ for $Q\in \mathrm{Max}(R)\setminus \{M,P\}$.
\end{lemma}

\section {First properties of extensions of length 2}

Next Theorem allows us to only  speak  of length 2 extensions.

\begin{theorem}\label{2.1} Let $R\subset S$ be a non-minimal extension. The following statemens are equivalent:
\begin{enumerate}
\item    $R\subset S$ is a minimal pair.
\item     $R\subset S$ is a  co-minimal pair.
\item  $\ell[R,S]=2$. 
\end{enumerate}

Hence,  an extension $R\subset  S$ with   $|[R,S]|=3$ has length    $2$.
 
\end{theorem}

\begin{proof}  (1) $\Rightarrow$ (2) Let $T\in ]R,S[$ be such that $R\subset T$ is  minimal. If $T\subset S$ is not  minimal, there exists a tower of extensions $R\subset T\subset T'\subset S$. But the assumption gives that $T\in]R,T'[$ with $R\subset T'$  minimal, a contradiction and then $T\subset S$ is minimal.

(2) $\Rightarrow$ (1) and (3) Let $T\in ] R,S[$ be such that $T\subset S$ is  minimal. Assume that $R\subset T$ is not  minimal. There is a tower of extensions $R\subset T'\subset T\subset S$. As $T\in ]T',S[$  contradicts  $T'\subset S$  minimal,  it follows that $R\subset T$ is minimal.  In particular, any maximal chain from $R$ to $S$ has length 2 and $\ell[R,S]=2$.

(3) $\Rightarrow$ (1) Assume that $\ell[R,S]=2$. In particular, $R\subset S$ is  non-minimal. Let $T\in ]R,S[$, then, $R\subset T\subset S$ is a maximal chain, so that $R\subset T$ is a minimal extension.

The last result is obvious.
\end{proof}

\begin{proposition}\label{2.2} Let $R\subset S$ be a length 2 extension. Then:

 (1) Either $R\subset S$ is pointwise minimal (equivalently, co-pointwise minimal) or there exists $x\in S$ such that $S=R[x]$,  that is $S $ is simple and these conditions are mutually exclusive.

(2) $R\subset S$ is a quasi-Pr\"ufer  extension; so that either $\overline R\in\{R,S\}$, or $R\subset \overline R$ and $\overline R\subset S$ are minimal.
\end{proposition}

\begin{proof} (1) Obviously, $R\subset S$ is not  minimal. Asssume that $S\neq R[x]$ for any $x\in S$. Let $x\in S\setminus R$ and set $T:=R[x]$. Then, $T\in]R,S[$ so that $R\subset T$ and $T\subset S$ are  minimal   (Theorem~\ref{2.1}). By definition, $R\subset S$ is a pointwise minimal extension and a  co-pointwise minimal extension.  

 The first part of (2) is \cite[Corollary 3.4]{Pic 5}. For the second part, consider the tower $R\subseteq \overline R\subseteq S$ whose length  is $\leq 2$. 
\end{proof}

\begin{proposition}\label{2.3} If $R\subset S$ has length 2, then,  $|\mathrm{Supp} (S/R)|\leq 2$.
\end{proposition}

\begin{proof} By definition, $R\subset S$ is an FCP extension. It follows from (Proposition ~\ref{1.14}) that $|\mathrm{Supp}(S/R)|\leq 2$. In particular, for any $T\in]R,S[$, we have $\mathrm{Supp}(S/R)=\{\mathcal{C}(R,T),\mathcal{C}(T,S)\cap R\}$. 
\end{proof}
 We recall  the characterization of co-pointwise minimal extensions gotten in \cite[Theorem 3.2, Proposition 3.9 and Corollary 5.9]{Pic 7}:

\begin{proposition}\label{2.4} Let $R\subset S$ be a ring extension. Then $R\subset S$ is a co-pointwise minimal extension if and only if there is a maximal ideal $M$ of $R$ such that $MS = M$ and one of the following mutually exclusive conditions is satisfied, where $k:=R/M$ and  $p:=\mathrm{c}(k)$.

\begin{enumerate}
\item  $k= \mathbb{Z}/2\mathbb{Z}$ and  $S/M \cong k^3.$
\item  $S/M$ is a field, $x^p\in R$ for each $x\in S$ and $[S/M:k] = p^2$.
\item  $S/M \cong k[X,Y]/(X^2,XY,Y^2).$ 
\end{enumerate}

In each case, $S=R[x,y]$, where $\{x,y\}$ is a minimal system of generators and $\ell[R,S]=2$.
\end{proposition}

Taking into account the above results, to obtain a complete characterization
of length 2 extensions, we need only to study those who are simple.  In view of (Proposition~\ref{2.3}), each possible cardinality of  the support is examined in different sections. 

\section {Length 2 extensions   whose support has two elements}

 We remark that an extension whose support has two elements is not pointwise minimal  (Theorem ~\ref{1.2}) and is necessarily simple.

\begin{proposition}\label{3.1} Let $R\subset S$ be an extension such that $|\mathrm{Supp}(S/R)|=2$ and $\mathrm{Supp}(S/R)\subseteq\mathrm{Max}(R)$. The following conditions are equivalent:

(1) $\ell[R,S]=2$. 

(2) $R\subset S$ is locally minimal.

(3) $|[R,S]|=4$. 

If these conditions hold, then $R\subset S$ is simple.
\end{proposition}

\begin{proof} Set $\mathrm{Supp}(S/R):=\{M_1,M_2\}$.  

(1) $\Rightarrow$ (2) By   (Lemma ~\ref{1.15}), there exists $T_i\in]R,S[$ such that $\mathcal{C}(R,T_i)=M_i$ for $i=1,2$. Of course, $T_1\neq T_2$ with $R\subset T_i$ and $T_i\subset S$ minimal for $i=1,2$. Since $(T_i)_{M_j}=R_{M_j}$, it follows that $(T_i)_{M_j}\neq S_{M_j}$ for $i\neq j$, giving that $R_{M_j}\subset S_{M_j}$ is minimal for $j=1,2$, whence (2) holds.

(2) $\Rightarrow$ (3) Let $T\in]R,S[$. Then, $T_P=R_P=S_P$ for each $P\in\mathrm{Spec}(R)\setminus\{M_1,M_2\}$. Let $i\in\{1,2\}$. Since $R_{M_i}\subset S_{M_i}$ is minimal, we get either $T_{M_i}=R_{M_i}\ (*)$, or  $T_{M_i}\neq R_{M_i}\ (**)$, and in this case,  $T_{M_i}=S_{M_i}$.  In case $(*)$, we cannot have $T_{M_j}=R_{M_j}$, so that $T_{M_j}=S_{M_j}$. Then, there is at most one $T\in]R,S[$ satisfying $(*)$, and, in the same way, there is at most one $T'\in]R,S[$ satisfying $(**)$. Hence, $|[R,S]|\leq 4$. In particular, $R\subset S$ has FMC and, using (Lemma~\ref{1.15}), there exist $T_i\in]R,S[$ such that $\mathcal{C}(R,T_i)=M_i$ for $i=1,2$, so that $|[R,S]|= 4$.  

(3) $\Rightarrow$ (1) Assume  that $|[R,S]|=4$ and set $[R,S]=\{R,T,T',S\}$. Since $R\subset T,T'\subset S$, either $T$ and $T'$ are comparable, or they are incomparable. In this last case there are only two maximal chains  of length 2 from $R$ to $S$,  because $|[R,S]| =4$. If  $T$ and $T'$ are comparable, we have a chain of length 3, containing necessarily two minimal extensions $R_1\subset R_2\subset R_3$, with $R_1,R_2,R_3\in[R,S]$ such that $N_1:=\mathcal{C}(R_1,R_2)$ and $N_2=\mathcal{C}(R_2,R_3)$,  satisfying (for instance) $M_i=N_i\cap R$ for $i=1,2$. We claim that $N_2\cap R_1\not\subseteq N_1$. Deny, then  $N_2\cap R\subseteq N_1\cap R$ entails that $M_1$ and $M_2$ are comparable, a contradiction. Using again the Crosswise Exchange Lemma, we get that there is some $R'_2\in [R_1,R_3]$ such that $R_1\subset R'_2$ and $ R'_2\subset R_3$ are minimal  with $R'_2\neq R_2$, so that $|[R,S]|>4$, a contradiction. Therefore, $\ell[R,S]=2$.
\end{proof}

\begin{proposition}\label{3.2} Let $R\subset S$ be a length 2 extension. Assume  that  $|\mathrm{Supp} (S/R)|= 2$ and $\mathrm{Supp} (S/R)\not\subseteq \mathrm{Max} (R)$. Then, $R\subset S$ is a Pr\"ufer  extension, $|[R,S]|=3$ and $R\subset S$ is a simple extension.
\end{proposition}

\begin{proof} Set $\mathrm{Supp}(S/R)=\{M,P\}$. Since a support contains necessarily a maximal ideal, we can assume that $M\in\mathrm{Max}(R)$. Moreover, since $P\not\in\mathrm{Max}(R)$, we have $P\subset M$ because any maximal ideal of $R$ containing $P$ belongs to $\mathrm{Supp}(S/R)$. Observe that $R\subset S$ is quasi-Pr\"ufer  (Proposition~\ref{2.2}).  We have either $\overline R\in\{R,S\}$ or $R\subset \overline R$ and $\overline R\subset S$ both minimal. Assume  $\overline R=S$, then $R\subset S$ is an integral extension and $\mathrm{Supp}(S/R)\subseteq \mathrm{Max}(R)$   by \cite[Theorem 3.6(b)]{DPP2}, a contradiction. If $R\subset \overline R$ and $\overline R\subset S$ are both minimal, then $\mathcal{C}(\overline R,S)\in \mathrm{Max}(\overline R)$, giving again $\mathrm{Supp}(S/R)\subseteq \mathrm{Max}(R)$, still a contradiction. The only possible case is $\overline R=R$, so that $R\subset S$ is a Pr\"ufer  extension. 
Since the map $[R,S]\to[R_M,S_M]$ defined by $T\mapsto T_M$ is a poset isomorphism \cite[Lemma 3.5, Theorem 3.6]{DPP2},    we get that  $|[R,S]|=3$  \cite[Theorem 6.10(b)]{DPP2}.
\end{proof}

\begin{corollary}\label{3.23} Let $R\subset S$ be  a Pr\"ufer  extension. Then, $\ell[R,S]=2$ if and only if  $|\mathrm{Supp} (S/R)|= 2$, in which case $R\subset S$ is  simple.
\end{corollary}

\begin{proof} Assume that $\ell[R,S]=2$, so that $|\mathrm{Supp} (S/R)|\leq  2$ by Proposition \ref{2.3}. If $|\mathrm{Supp} (S/R)|= 1$, we may assume that $(R,M)$ is a local ring with $\mathrm{Supp} (S/R)= \{M\}$. Then, $R\subset S$ is minimal by \cite[Theorem 6.10 and Proposition 6.12]{DPP2}, a contradiction. It follows that  $|\mathrm{Supp} (S/R)|= 2$.

Conversely, assume that $|\mathrm{Supp} (S/R)|= 2$. Since  $R\subset S$ is Pr\"ufer, it is a normal pair,  $R\subset S$ has FIP   \cite[Proposition 6.9]{DPP2} and then $\ell[R,S]=2$   \cite[Proposition 6.12]{DPP2}.
\end{proof}

An obvious example is a valuation domain $R$ of dimension 2 with maximal ideal $M$. Then, for $T:=R_P$,  where $\mathrm{Spec}(R)=\{P,M\}$ 
 and $S$ being the quotient field of $R$, we get that $|\mathrm{Supp} (S/R)|= 2$, so that $[R,S]=\{R,T,S\},\ \ell[R,S]=2$ and $|[R,S]|=3$. 

\section{$M$-crucial extensions of length 2} 

The remaining case occurs for $|\mathrm{Supp} (S/R)|= 1$, which  means that the extension is $M$-crucial. Next result shows that in the rest of the paper, we can reduce our proofs to the case a local ring.

\begin{proposition}\label{3.11}  An extension $R\subset S$  with $\mathrm{MSupp} (S/R)= \{M\}$ verifies: 

(1) The map $\varphi : [R,S]\to [R_M,S_M]$ defined by $T\mapsto T_M$ for any $T\in[R,S]$ is a poset isomorphism.  Therefore,  $\ell[R,S]=\ell[R_M,S_M]$ and $|[R,S]|=|[R_M,S_M]|$.

(2) Let $R=R_0\subset R_1\subset\cdots\subset R_{n-1}\subset R_n=S$ be a maximal chain (whence $R_i\subset R_{i+1}$ is  minimal  for each $i=0,\ldots,n-1$). Then so is $R_M\subset (R_1)_M\subset\cdots\subset (R_{n-1})_M\subset S_M$ and $(R_i)_M\subset (R_{i+1})_M$ is  minimal  of the same type as $R_i\subset R_{i+1}$ for each $i=0,\ldots,n-1$.

(3)   $R\subset S$ has FCP (resp. FIP) if and only if $R_M\subset S_M$ has FCP (resp. FIP).

(4) Assume that $R\subset S$ is an integral extension. Then, $\mathrm{Max}(S_M)=\{NR_M\mid N\in\mathrm{Max}(S),\ N\cap R=M\}=\{NR_M\mid N\in\mathrm{V}(MS)\}$. Moreover, if $R\subset S$ is  finite, then $(R_M:S_M)=(R:S)_M$.
\end{proposition}

\begin{proof} 
(1)    
  $\varphi$ is obviously a poset isomorphism.  This gives the equalities for the lengths and cardinalities.

(2) Let $i\in\{0,\ldots,n-1\}$. Since $\mathrm{MSupp} (S/R)=\{M\}$, we get that $\mathcal{C}(R_i,R_{i+1})\cap R\subseteq M$, so that $M\in\mathrm{MSupp} (R_{i+1}/R_i)$ and $(R_i)_M\subset (R_{i+1})_M$ is  minimal    \cite[Lemme 1.3]{FO}. It is then enough to check the characterization of each type of minimal extensions to see that $(R_i)_M\subset (R_{i+1})_M$ is a minimal extension of the same type as $R_i\subset R_{i+1}$.

(3) is  obvious since $\varphi$ is bijective and (4) comes from properties of localizations.
\end{proof}

 In this section, we characterize  length 2  extensions that are $M$-crucial ({\it i.e} when  $\mathrm{Supp} (S/R)=\{M\}$). If a co-pointwise minimal extension is involved, we recall the ad hoc result. By  an exhaustive process,  we achieve  a characterization of   length 2 simple  $M$-crucial extensions.

\begin{proposition}\label{3.3} Let  $R\subset S$ be a non-integral  $M$-crucial extension. Then  the following conditions are equivalent:

(1) $\ell [R,S]= 2$.

(2)  $R\subset \overline R$ and $\overline R\subset S$ are minimal. 

(3)  $|[R,S]|=3$.

If the above equivalent conditions hold,  $R\subset S$ is  simple.  Moreover, $R\subset S$ is quasi-Pr\"ufer and not Pr\"ufer.
\end{proposition}

\begin{proof} (1) $\Rightarrow$ (2)  $R\subset S$ is not a Pr\"ufer  extension in view of \cite[Theorem 6.3 and Proposition 6.12]{DPP2}. It follows that $\overline R\neq R,S$ so that $R\subset \overline R$ and $\overline R\subset S$ are minimal extensions. 

(2) $\Rightarrow $ (3) $R\subset S$ has FCP in view of \cite[Theorem 3.13]{DPP2}. Assume that there exists $T\in]R,S[\setminus\{\overline R\}$.  Then, $T$ and $\overline R$ are not comparable, and we may assume that $R\subset T$  is  minimal. Since $R\subset \overline R$ is minimal, $R\subset T$ cannot be minimal integral and $R\subset T$ cannot be minimal Pr\"ufer  (Lemma~\ref{1.12}). Hence, we get a contradiction and $|[R,S]|=3$.

(3) $\Rightarrow $ (1) Obvious.

Moreover, $R\subset S$ is  simple beause of (Proposition~\ref{2.2}) and  the classification of co-pointwise minimal extensions of (Proposition~\ref{2.4}).
\end{proof}

\begin{corollary}\label{3.31}  Let  $R\subset S$ be a non-integral non-minimal quasi-Pr\"ufer $M$-crucial extension with $|\mathrm{Supp}_{\overline R}(S/\overline R)|=1$. Then, $\ell[R,S]=2$ if and only if $R\subset \overline R$ is minimal.  
\end{corollary}

\begin{proof} Since $R\subset S$ is $M$-crucial, $\mathrm{Supp}(S/R)=\{M\}$. Moreover, $S\neq\overline R$, so that $\overline R\subset S$ is Pr\"ufer with $\mathrm{Supp}_{\overline R}(S/\overline R)$ finite. Then, $\overline R\subset S$ has FCP  \cite[Proposition 1.3]{Pic 5}. It follows that $\overline R\subset S$ has  FIP  \cite[Proposition 6.9]{DPP2} and  is minimal in view of \cite[Proposition 6.12]{DPP2}. The equivalence is then obvious. 
\end{proof}

Here is an explicit example \cite[Example 1, page 376]{MP}. It is enough to choose an order of algebraic integers $R$ such that $R\subset T$ is minimal inert with conductor $M=Rp$, where $T$ is the integral closure of $R$ and a PID. Setting $S:=\overline R_p$, we get that $T=\overline R$ is such that $R\subset\overline R$ and $\overline R\subset S$ are minimal, $R\subset S$ is a non-integral FCP $M$-crucial extension such that $|\mathrm{Supp}_{\overline R}(S/\overline R)|=1$ since $\mathrm{Supp}(S/\overline R)=\{M\}$. Then, $\ell[R,S]=2$.  Take for instance $R:=\mathbb {Z}[5\sqrt{-2}],\ \overline R=\mathbb {Z}[\sqrt{-2}]$ and $S:=(\overline R)_5$.

The last case to consider, which will be the more complicated, is when $R\subset S$ is an integral $M$-crucial extension of length 2.
 To get a  characterization of such extensions, we need the following recalls. 

\begin{definition}\label{3.4} An integral extension $R\subseteq S$ is called {\it infra-integral} \cite{Pic 2} (resp$.$; {\it subintegral} \cite{S}) if all its residual extensions $R_P/PR_P\to S_Q/QS_Q$, (with $Q\in\mathrm{Spec}(S)$ and $P:=Q\cap R$) are isomorphisms (resp$.$; and the spectral map $\mathrm{Spec}(S)\to\mathrm{Spec}(R)$ is bijective). An extension $R\subseteq S$ is called {\it t-closed} (cf. \cite{Pic 2}) if the relations $b\in S,\ r\in R,\ b^2-rb\in R,\ b^3-rb^2\in R$ imply $b\in R$. The $t$-{\it closure} ${}_S^tR$ of $R$ in $S$ is the smallest element $B\in[R,S]$ such that $B\subseteq S$ is t-closed and the greatest element $B'\in[R,S]$ such that $R\subseteq B'$ is infra-integral. An extension $R\subseteq S$ is called {\it seminormal} (cf. \cite{S}) if the relations $b\in S,\ b^2\in R,\ b^3\in R$ imply $b\in R$. These two properties are stable under the formation of  localizations. The {\it seminormalization} ${}_S^+R$ of $R$ in $S$ is the smallest element $B\in[R,S]$ such that $B\subseteq S$ is seminormal and the greatest  element $ B'\in[R,S]$ such that $R\subseteq B'$ is subintegral.

The canonical decomposition of an arbitrary ring extension $R\subset S$ is $R \subseteq {}_S^+R\subseteq {}_S^tR \subseteq \overline R \subseteq S$.
\end{definition}

 Next proposition gives the link between the elements of the canonical decomposition and minimal extensions.

\begin{proposition}\label{3.5} \cite[Lemma 3.1]{Pic 4} Let there be an integral extension $R\subset S$ and a maximal chain $\mathcal C$ of $ R$-subextensions of $S$, defined by $R=R_0\subset\cdots\subset R_i\subset\cdots\subset R _n= S$, where each $R_i\subset R_{i+1}$ is  minimal.  The following statements hold: 

\begin{enumerate}
\item $R\subset S$ is subintegral if and only if each $R_i\subset R_{i+1}$ is  ramified. 

\item $R\subset S$ is seminormal and infra-integral if and only if each  $R_i\subset R_{i+1}$ is decomposed. 

\item $R \subset S$ is t-closed if and only if  each $R_i\subset R_{i+1}$ is inert. 
\end{enumerate}
\end{proposition}

\begin{proof}  \cite[Lemma 3.1]{Pic 4} asserts  that (3) holds and gives the infra-integral part of (1) and (2).
 Now (1) is clear since  we deal with a bijective  spectral map. If $R\subset S$ is seminormal, 
  $(R:S)$ is a finite intersection of  maximal ideals of $S$ (resp. $R_{i+1}$)  by an easy generalization of \cite[Proposition 4.9]{DPP2}, giving (2). 
\end{proof} 

\begin{definition}\label{4.5} Let $R\subset S$ be an integral extension of length 2. The tower $R\subseteq {}_S^+R\subseteq {}_S^tR\subseteq S$ shows  that $|\{R,{}_S^+R,{}_S^tR,S\}|\leq 3$, so that at least two of these elements are equal. This gives the six following cases:

(a) $R={}_S^+R={}_S^tR\neq S$.  \, \, \,  (b) $R={}_S^+R\neq{}_S^tR\neq S$.

(c) $R={}_S^+R\neq{}_S^tR= S$.  \, \, \,  (d) $R\neq{}_S^+R={}_S^tR\neq S$.

(e) $R\neq{}_S^+R\neq{}_S^tR= S$. \; \, \,  (f) $R\neq{}_S^+R={}_S^tR =S$.
\end{definition} 

(Proposition~\ref{2.4}) recalls that there are three types of co-pointwise minimal extension $R\subseteq S$: In (1) $R\subset S$ is seminormal and infra-integral ($R={}_S^+R$ and $S={}_S^tR$): case (c), in (2) $R\subset S$ is inert ($R={}_S^+R={}_S^tR$): case (a)  and in (3) $R\subset S$ is  subintegral ($S={}_S^+R={}_S^tR$): case (f). In particular, if $R\subset S$ is co-pointwise minimal, then $\{{}_S^+R,{}_S^tR\}\subseteq\{R,S\}$.  To consider integral extensions of length 2 which are not co-pointwise minimal, we begin to consider  the cases of an extension $R\subset S$ where either the seminormalization or the t-closure is different from $R$ and $S$, that is cases (b), (d) and (e). Such extensions are simple. 

We first  consider  cases (b) and (d) which give the same result.

 \begin{proposition}\label{3.6} Let $R\subset S$ be  an integral  $M$-crucial extension. The following are equivalent:

(1)   ${}_S^tR\neq R,S$ and  $\ell[R,S]=2$.

(2) $ R\subset {}_S^tR$ and ${}_S^tR \subset S$ are minimal extensions and $|[R,S]|=3$.

If these conditions hold, then $R\subset S$ is a simple extension.
\end{proposition}

\begin{proof} (1) $\Rightarrow$ (2)  Assume that ${}_S^tR\neq R,S$ and $\ell[R,S]=2$. The first assertion is obvious. In particular, $R\subset{}_S^tR$ is a minimal infra-integral extension and ${}_S^tR\subset S$ is a minimal inert extension. Assume that there exists $T\in]R,S[\setminus\{{}_S^tR\}$, so that $R\subset T$ is a minimal extension. Since $R\subset{}_S^tR$ is minimal, $R\subset T$ cannot be minimal infra-integral, so that $R\subset T$ is minimal inert. Moreover, we get that $S=T({}_S^tR)$. Then, an appeal to \cite[Propositions 7.1 and 7.4]{DPPS} shows that $\ell[R,S]>2$, a contradiction. To conclude, $|[R,S]|=3$.

(2) $\Rightarrow$ (1) is obvious.

Finally, (Proposition \ref{2.4}) shows that $R\subset S$ cannot be a co-pointwise minimal extension, so that $R\subset S$ is simple. 
\end{proof}

\begin{corollary}\label{3.61}  Let $(R,M)$ be a local ring and $R\subset S$  an integral   $M$-crucial extension such that ${}_S^tR\neq R,S$ and ${}_S^tR\subset S$ is minimal.   Then, $\ell[R,S]=2$ if and only if one of the following condition holds:

\noindent $\mathrm{(1)}$  $|\mathrm {Max}(S)|=1$ and ${\mathrm L}_R(N/M)=1$ for $\mathrm {Max}(S):=\{N\}$. 

\noindent  $\mathrm{(2)}$   $|\mathrm  {Max}(S)|=2$ and $M$ is  an intersection of two maximal ideals of $S$.

If $R\subset S$ satisfies one of these conditions, then $|[R,S]|=3$. 
\end{corollary}

\begin{proof} Setting $T:={}_S^tR$, we observe that $T\subset S$ is minimal inert and $R\subset T$ is infra-integral.

Assume that $\ell[R,S]=2$. Then $R\subset T$ is minimal either ramified or decomposed. Moreover, $R\subset S$ has FCP.

If $R\subset T$ is ramified,  $T$ and $S$ are local rings, so that $|\mathrm {Max}(S)|=1$. Let $N$ be the maximal ideal of $T$. Then, $N=(T:S)$ and is the maximal ideal of $S$. Moreover, ${\mathrm L}_R(N/M)=1$ by \cite[Lemma 5.4]{DPP2}.

If $R\subset T$ is decomposed, then $|\mathrm{Max}(T)|=2$. Let $\{M_1,M_2\}:=\mathrm{Max}(T)$. Then, $M=M_1\cap M_2=M_1M_2\ (*)$, and $(T:S)$ is one of the $M_i$. Assume that $(T:S)=M_1$ for instance, then $M_1$ is a maximal ideal of $S$, the other being $M'_2:=M_2S$  \cite[Lemma 2.4]{DPP2}. By $(*)$, we get that $M=M_1SM_2=M_1M'_2=M_1\cap M'_2$, which is a radical ideal in $S$.

Conversely, assume that $|\mathrm {Max}(S)|={\mathrm L}_R(N/M)=1$ for $\mathrm {Max}(S)=\{N\}$.  Since $T\subset S$ is minimal inert, we get that $|\mathrm {Max}(T)|=1$ and  $\mathrm {Max}(T)=\{N\}$ because $N=(T:S)$. But $R\subset T$ is  subintegral gives that $R/M\cong T/N$, whence $T=R+N\neq R$ since  ${\mathrm L}_R(N/M)=1$. Hence, $T=R+Rx$ for some $x\in N\setminus M$ such that $N=M+Rx$. So, we have $M\subseteq M+Mx\subseteq M+Rx=N$ which leads to either $M=M+Mx\ (*)$ or $M+Mx=M+Rx\ (**)$. In case $(*)$, we get that $Mx\subseteq M$ so that $M=(R:T)$. In case $(**)$, we get that $x\in M+Mx$, so that there exist $m,m'\in M$ such that $x=m+m'x$, which implies $(1-m')x=m\in M$. But $1-m'$ is a unit in $R$ so that $x\in M$, a contradiction.   It follows that only case $(*)$ holds. Now, $T=R+Rx$ and $R/M$ is Artinian give that $R\subset T$ has FCP by \cite[Theorem 4.2]{DPP2}. It follows that  $R\subset T$ is minimal ramified by \cite[Lemma 5.4]{DPP2}. Moreover, $|[R,S]|=3$, because there does not exist $T'\in]R,S[\setminus\{ T\}$  \cite[Proposition 7.4]{DPPS}. Then, use (Proposition~\ref{3.6}) to get $\ell[R,S]=2$.

Assume now that $|\mathrm{Max}(S)|=2$ and $M$ is an  intersection of two maximal ideals of $S$. Since $T\subset S$ is minimal inert, we get that $|\mathrm {Max}(T)|=2$. Then, $M$ is also an  intersection of two maximal ideals of $T$  and $M=M_1M_2=M_1\cap M_2$, where $\mathrm{Max}(T)=\{M_1,M_2\}$.  
Moreover, we infer from $R/M\cong T/M_i$, for $i=1,2$  that $T/M\cong T/M_1\times T/M_2\cong (R/M)^2$ and then $R/M\subset T/M$  is minimal decomposed, and so is $R\subset T$.  At last, $|[R,S]|=3$, because there does not exist $T'\in]R,S[\setminus\{T\}$ in view of \cite[Proposition 7.1]{DPPS}. Then, use (Proposition ~\ref{3.6}) to get $\ell[R,S]=2$.
\end{proof}

\begin{corollary}\label{3.62}  Let $R\subset S$ be an integral   $M$-crucial extension such that ${}_S^tR\neq R,S$ and ${}_S^tR\subset S$ is minimal.   Then, $\ell[R,S]=2$ if and only if one of the following condition holds:
\begin{enumerate}
\item $|\mathrm V(MS)|=1$ and ${\mathrm L}_R(N/M)=1$ for $\mathrm V(MS)=\{N\}$. 
\item $|\mathrm V(MS)|=2$ and $M= N_1\cap N_2$, where $N_1, N_2 \in \mathrm{Max}(S)$.

\end{enumerate} 

If $R\subset S$ satisfies one of these conditions, then $|[R,S]|=3$. 
\end{corollary}

\begin{proof} Because of (Proposition~\ref{3.11}), we get that $\ell[R,S]=\ell[R_M,S_M]$, 
 $|[R,S]|=|[R_M,S_M]|$ and $\mathrm{Max}(S_M)=\{NR_M\mid N\in\mathrm{V}(MS)\}$, giving $|\mathrm{Max}(S_M)|=|\mathrm{V}(MS)|$. 

Assume that $|\mathrm{Max}(S_M)|=1$, so that there is a unique $N\in\mathrm{Max}(S)$ lying over $M$. In view of the localization formula of Northcott \cite[Theorem 12, page 166]{N}, we have ${\mathrm L}_R(N/M)={\mathrm L}_{R_M}(N_M/M_M)$.

Assume that $|\mathrm{Max}(S_M)|=2$. There exist $N_1,N_2\in\mathrm{Max}(S)$ lying over $M$, such that $\mathrm{Max}(S_M)=\{(N_1)_M,(N_2)_M\}$ which are lying over $M_M$. If $M_M$ is a radical ideal of $S_M$,  intersection of two ideals of $S$, 
then $M_M=(N_1)_M\cap(N_2)_M=(N_1\cap N_2)_M$, giving $M=N_1\cap N_2$.

It is now enough to translate the conditions of (Corollary~\ref{3.61})
\end{proof}

\begin{example} \label{3.622} We illustrate (Corollary~\ref{3.61}) with examples due to D. Dobbs and J. Shapiro. 

(1) \cite[Remark 3.4 (h)]{DS} Let $K\subset L$ be a minimal field extension of degree 2, so that there exists $y\in L$ such that $L=K[y]=K+Ky$. Set $S:=L[X]/(X^2)$ and let $x$ be the class of $X$ in $S$, so that $x^2=0$ and $S=L+Lx=K+Ky+Kx+Kxy$. Set $R:=K+Kx$ and $T:=K+Lx=K+Kx+Kxy=R+Rxy=R[xy]$ (because $x^2=0$). Since $K\subset R$ is obviously a minimal ramified extension with crucial maximal ideal $0$, it follows that $R$ is a local ring with maximal ideal $M=Kx$. But, $Mxy=Kx^2y=0\subseteq M$ and $(xy)^2=0$ give that $R\subset T$ is minimal ramified, so that $T$ is a local ring with maximal ideal $M':=M+Rxy=Kx+Kxy$. In the same way, $L\subset S$ is also a minimal ramified extension, so that $S$ is a local ring with maximal ideal $N:=Lx=Kx+Kxy=M'$. Now, $T/M'=T/N\cong K$ and $S/M'\cong L$ imply that $T\subset S$ is  minimal inert . Moreover, $R\subset S$ is  $M$-crucial. To end, ${\mathrm L}_R(N/M)=\dim_K((Kx+Kxy)/(Kx))=1$ gives that $\ell[R,S]=2$  and $|[R,S]|=3$ by 
(Corollary~\ref{3.61}(1)). In fact, $S=R[y]$. 

(2) \cite[Remark 3.4 (c)]{DS} Let $K\subset L$ be a minimal field extension. Set $R:=K,\ T:=K\times K$ and $S:=K\times L$. Then, $R\subset T$ is a minimal decomposed extension with crucial maximal ideal $M:=0$. Let $N:=K\times 0$, which is a maximal ideal of $T$ lying above $M$ and the conductor of $T\subset S$. Since $S/N\cong L$, we get that $N$ is also a maximal ideal of $S$. Let $N':=0\times K$, which is the other maximal ideal of $T$, so that $N\cap N'=NN'=0=M$. Now, $T\subset S$ is  minimal inert  because $T/N\cong K$. Then, ${}_S^tR=T$. The (only) maximal ideal of $S$ lying above $N'$ is $P:=0\times L$. It follows that $|\mathrm{Max}(S)|=2$. Moreover, $0=NP=N\cap P$ is a radical ideal of $S$, giving that $\ell[R,S]=2$ and $|[R,S]|=3$ by (Corollary~\ref{3.61}(2)). In fact, $S=R[y]$ where $y=(0,z)$ is such that $L=K[z]$.  

In the two previous examples, $R\subset S$ is  not  pointwise minimal  because $M\neq (R:S)$ in (1) and by \cite[Proposition 4.14]{Pic 7} in (2).

\end {example}

 The following tables summarize  characterizations  of ramified and decomposed extensions gotten in (Theorem~\ref{1.3}) and (Proposition~\ref{3.5}). In each table, the two lines of each column are equivalent:
 \vskip 0,2cm
 Table T$_1$: $R\subset S$ is a minimal extension with conductor $M$.

\vskip 0,2cm
 
  \centerline{\begin{tabular} {|c|c|}
 \hline
 $R\subset S$ ramified                            & $R\subset S$ decomposed \\
 \hline
 $\exists x\in S$ such that  $S=R[x]$,\    & $\exists x\in S$ such that $S=R[x]$,\\   
 $x^2\in M$ and $xM\subseteq M$         & $x^2-x\in M$ and $xM\subseteq M$ \\
  \hline
 \end{tabular}}
  \vskip 0,2cm
  
 Table T$_2$: $R\subset T$ and $T\subset S$ are minimal and $R\subset S$ is infra-integral}

\vskip 0,2cm
 
 \centerline{\begin{tabular} {|c|c|c|}
 \hline
$R\subset S$ subintegral & $R\subset S$ seminormal   & ${}_S^+R\neq R,S$ \\
\hline
$R\subset T$ ramified      & $R\subset T$ decomposed & $R\subset{}_S^+R$  ramified  \\
$T\subset S$ ramified      & $T\subset S$ decomposed & ${}_S^+R\subset S$  decomposed  \\
\hline
 \end{tabular}}
  \vskip 0,2cm

\begin{proposition}\label{3.623} Let $(R,M)$ be a local ring and $R\subset S$ an infra-integral  $M$-crucial extension and $N:=\mathrm{Rad} (S) $. Then, $\ell[R,S]= 2$   if and only if ${\mathrm L}_R(N/M) + |\mathrm {Max}(S)|= 3$.
\end{proposition}

\begin{proof} Assume that $\ell[R,S]= 2$. Then, $R\subset S$ has FCP. Now, use \cite[Lemma 5.4]{DPP2} to get the wanted equality.

Conversely, assume  ${\mathrm L}_R(N/M)+|\mathrm {Max}(S)|=3$. Since $|\mathrm {Max}(S)|\geq 1$,  we have  three cases: (1) ${\mathrm L}_R(N/M)=0$ and $|\mathrm {Max}(S)|=3$, (2) ${\mathrm L}_R(N/M)=1$ and $|\mathrm {Max}(S)|=2$, (3) ${\mathrm L}_R(N/M)=2$ and $|\mathrm {Max}(S)|=1$.

(1) Assume that ${\mathrm L}_R(N/M)=0$ and $|\mathrm {Max}(S)|=3$. We get $N=M=M_1 \cap M_2 \cap M_3$, where $\{M_1, M_2, M_3\}=\mathrm {Max}(S)$. Since $R\subset S$ is infra-integral, $R/M\cong S/M_i$ for $i\in\{1, 2, 3\}$. Then, $S/M\cong \prod_{i=1}^3(S/M_i)\cong (R/M)^3$ entails that $R/M\subset S/M$ is a seminormal  infra-integral FCP extension of length 2  \cite[Proposition 4.15]{DPP3} and \cite[Lemma 5.4]{DPP2}, and so is $R\subset S$. To conclude,  $\ell[R,S]=2$.

(2) Assume that ${\mathrm L}_R(N/M)=1$ and $|\mathrm {Max}(S)|=2$, so that $N=M_1 \cap M_2 $, where $\{M_1, M_2\}=\mathrm {Max}(S)$. Moreover, there exists some $x\in N\setminus M$ such that $N=M+Rx$. It follows that $M\subseteq M+Mx\subseteq M+Rx=N$. From ${\mathrm L}_R(N/M)=1$ we deduce that either $M+Mx=M\ (*)$ or $M+Mx=M+Rx\ (**)$. We claim that $(**)$ cannot occur. Deny, then $x\in M+Mx$ implies that there are $m,m'\in M$ with $x=m+m'x$.  Therefore,  $(1-m')x=m$ and since $1-m'$ is a unit of $R$, we get $x\in M$, a contradiction. Then, only $(*)$ holds, so that $Mx\subseteq M$. Set $T:=R+N=R+Rx\in[R,S]$,  then $N$ is an ideal of $T$. Now, $T/N=(R+N)/N\cong R/M$ shows that $N\in\mathrm {Max}(T)$ and is the only maximal of $T$ since $M_1, M_2$ lie over $N$. To end ${\mathrm L}_R(T/R)={\mathrm L}_R((R+N)/R)={\mathrm L}_R(N/(R\cap N))={\mathrm L}_R(N/M)=1$ shows that $R\subset T$ is a minimal, necessarily ramified extension. 

Since $R\subset S$ is infra-integral, we get $R/M\cong T/N\cong S/M_i$ for $i=1,2$. But $S/N\cong S/M_1\times S/M2\cong (T/N)^2$ implies that $T/N\subset S/N$ identifies with $T/N\subset (T/N)^2$, which is  minimal decomposed, and so is $T\subset S$.  Then, $R\subset S$ has FCP since it  is a tower of two minimal finite extensions    \cite[Theorem 4.2]{DPP2}. To conclude $\ell[R,S]=2$ by \cite[Lemma 5.4]{DPP2} since $R\subset S$ is infra-integral.

(3) Assume that ${\mathrm L}_R(N/M)=2$ and $|\mathrm {Max}(S)|=1$, so that $\{N\}=\mathrm {Max}(S)$. Then, $S$ is a local ring and $R\subset S$ is subintegral. Moreover, $S/N\cong R/M$ gives that $S=R+N$. It follows that ${\mathrm L}_R(S /R)={\mathrm L}_R((R+N) /N)={\mathrm L}_R(N /(R\cap N))={\mathrm L}_R(N/M)=2$, which shows that $R\subset S$ has FCP. At last, \cite[Lemma 5.4]{DPP2} shows that $\ell[R,S]=2$.
 \end{proof}

\begin{proposition}\label{3.63} If  $R\subset S$ is an infra-integral  $M$-crucial extension and  $N:=\sqrt {MS} $, then $\ell[R,S]= 2$   if and only if ${\mathrm L}_R(N/M) + |\mathrm {V}(N)|= 3$.
\end{proposition}

\begin{proof} Because of (Proposition~\ref{3.11}), we get that $\ell[R,S]=\ell[R_M,S_M]$. If either $\ell[R,S]$ or $\mathrm {V}(N)$ is finite, we get that $N_M=\sqrt {MS_M} $ and $|\mathrm {V}(MS)|=|\mathrm {Max}(S_M)|$ as in (Corollary~\ref{3.62}). For the same reason, ${\mathrm L}_R(N/M)={\mathrm L}_{R_M}(N_M/M_M)$. To conclude, use (Proposition~\ref{3.623}).
\end{proof}

 We now  consider  case (e) of (Definition~\ref{4.5}).

\begin{proposition}\label{3.7} Let  $R\subset S$ be an infra-integral $M$-crucial extension such that $T:={}_S^+R\in ]R,S[$. Then, $\ell[R,S]=2$ if and only if $3\leq|[R,S]|\leq 4$ and $(R:S)= M$ when $|[R,S]|=4$.
In that case,  $R\subset {}_S^+R$ is  minimal ramified, ${}_S^+R \subset S$ is   minimal decomposed  and $R\subset S$ is  simple. Moreover, if $|[R,S]|=3$, then $(R:S)\neq M$.
\end{proposition}

\begin{proof} By (Proposition  ~\ref{3.11}), we can assume that $(R,M)$ is  local  since $\ell[R,S]=\ell[R_M,S_M]$ (resp. $|[R,S]|=|[R_M,S_M]|$).

Assume first that $\ell[R,S]=2$.  Since $R\subset S$ has FCP, the chain $R\subset T\subset S$ has length 2  (\cite[Lemma 5.4]{DPP2}), so that $R\subset T$ is  minimal ramified  and $T\subset S$ is  minimal decomposed. Therefore,  $|\mathrm{Max}(S)|=2$ and $R\subset S$ is  simple  by (Proposition~\ref{2.2}) and because of the classification of co-pointwise minimal extensions described in (Proposition~\ref{2.4}). We will discuss with respect to the emptiness  of $]R,S[\setminus\{ T\}$. If  $]R,S[\setminus\{ T\}= \emptyset$, then $|[R,S]|=3$. 

  Assume that there is some $T'\in]R,S[\setminus\{ T\}$, so that $R\subset T'$  is  minimal. Since $R\subset T$ is minimal, $R\subset T'$ cannot be minimal ramified, so that $R\subset T'$ is minimal decomposed. In particular, $|\mathrm{Max}(T')|=2$ and $T'$ has two maximal ideals $M_1,M_2$ satisfying $M=M_1M_2$. Moreover, $T'\subset S$ is necessarily minimal ramified since $|\mathrm{Max}(S)|=2$. Then, $(T':S)$ is for example $M_1$,   an ideal shared by $T'$ and $S$.  Hence, \cite[Lemma 2.4]{DPP2} yields that there is a unique maximal ideal in $S$ lying above $M_2$ which is $M'_2:=M_2S$. But, $M=M_1M_2=(M_1S)M_2=M_1(SM_2)=M_1M'_2$ is an ideal of $S$, so that $M=(R:S)$. Assume there is some $T''\in]R,S[\setminus\{ T',T\}$, so that $R\subset T''$  is  minimal. By the same reasoning as for $T'$, we get that $R\subset T''$ is minimal decomposed. Then, $S=T''T'$ gives that $R\subset S$ should be seminormal  \cite[Proposition 7.6]{DPPS} and (Proposition ~\ref{3.5}), a contradiction. Then, $|[R,S]|= 4$. 

We show the converse. 
 If $|[R,S]|=3$, then $\ell[R,S]=2$.  Assume now that $(R:S)=M$ with $|[R,S]|=4$ and set $[R,S]=\{R,T',T'',S\}$ with all elements distinct. Since $R\subset T',T''\subset S$, either $T'$ and $T''$ are comparable, or they are incomparable. In this last case we get two maximal chains of length 2 from $R$ to $S$  and $\ell[R,S]=2$. If $T''$ and $T'$ are comparable, we have a chain of length 3, with, for instance, $R\subset T'\subset T''\subset S$, where $R\subset T'$ is minimal ramified, and $T''\subset S$ is minimal decomposed. But, since $R\subset T'$ is ramified, there exists $x\in T'\setminus R$ such that $x^2\in M$ and $x\not\in M$ (Theorem \ref{1.3}). In particular, $M=(R:S)$ is not a radical ideal of $S$. Then, \cite[Lemma 17]{DPP4} yields again that there exists some $T'''\in[R,S]$ such that $T'''\subset S$ is minimal ramified, so that $T'''\neq T',T''$,  contradicting  the assumption.   
 
 We show that $(R:S)\neq M$ when $|[R,S]|=3$.  Deny and assume that $(R:S)= M$. Since $R\subset T$ is minimal ramified, there exists $x\in T\setminus R$ such that $T=R[x]$ with $x^2\in M$ and $x\not\in M$ (Theorem \ref{1.3}). In particular, $M$ is not a radical ideal of $S$. Then, \cite[Lemma 17]{DPP4} yields that there exists some $T'\in[R,S]$ such that $T'\subset S$ is minimal ramified, so that $|[R,S]|> 3$,  an absurdity. Hence, $(R:S)\neq M$.
\end{proof}

 We will see later on (Example~\ref{3.131}(4)) that the condition $(R:S)=M$ to have $\ell[R,S]=2$ is necessary when $|[R,S]|=4$. 

\begin{example} \label{3.71} We are going to give examples satisfying each condition of (Proposition~\ref{3.7}) using  (Table T$_1$).

(1) We gave this example to D. Dobbs in \cite[Example]{D}. Let $(R,M)$ be a SPIR such that $M^2=0$ with $M\neq 0$. There exists $t\in M$ such that $M=Rt$, so that $t^2=0$. Set $S:=R[Y]/(Y^2-Y) =R[y]$, where $y$ is the class of $Y$ in $S$. Set $x:=ty$ and $T:=R[x]$. It follows that we have the tower $R\subset T\subset S$ with $R\subset S$ an $M$-crucial extension such that $(R:S)=0\neq M$. We get that $T={}_S^+R$ with $R\subset T$ minimal ramified and $T\subset S$ is minimal decomposed. Then, $|[R,S]|=3$ and $[R,S]=\{R,T,S\}$ by (Proposition~\ref{3.7}).

(2) Let $(R,M)$ be a PVD with $M:=Rt$. Set $S:=R[X,Y]/(X^2,Y^2-Y,XY,tX,tY)$ and let $x$ (resp. $y$) be the class of $X$ (resp. $Y$) in $S$, so that $x^2=y^2-y=xy=tx=ty=0$. Then $S=R[x,y]=R+Rx+Ry$. Set $T:=R[x]$ and $T':=R[y]$. Then, $x^2=tx=0$ shows that $R\subset T$ is  ramified minimal with crucial maximal ideal $M$. This implies that there is a unique maximal ideal $N$ in $T$ lying above $M$ and it satisfies $N=M+Rx=Rt+Rx$. Now, $y^2-y=ty=0$ shows that $R\subset T'$ is  decomposed minimal, with crucial maximal ideal $M$.  Therefore, there are two maximal ideals $N_1,N_2$ in $T'$ lying above $M$ and satisfying (for example) $N_1=M+Ry=Rt+Ry$ and $N_2=M+R(1-y)=Rt+R(1-y)$. Moreover, $M=N_1N_2$. Since $T\neq T'$, we have $|[R,S]|\geq 4$. From \cite[Proposition 7.6 (a)]{DPPS}, we infer that $T\subset S$ is minimal decomposed because $NN_1\subseteq M$. In particular, $R\subset S=TT'$ is an infra-integral $M$-crucial extension which is not seminormal with $(R:S)=M$. Moreover $T={}_S^+R$ and $\ell[R,S]=2$ by \cite[Lemma 5.4]{DPP2}. Then, $|[R,T]|= 4$ and $[R,S]=\{R,T,T',S\}$ by (Proposition~\ref{3.7}).
\end{example}

We  next have  to consider  extensions $R\subset S$ whose seminormalizations and  t-closures are either $R$ or $S$, leading to the following cases of (Definition~\ref{4.5}) : $R\subset S$ is either subintegral  (case (f)), or seminormal infra-integral   (case (c)), or t-closed   (case (a)). In each of these cases, we  have either co-pointwise minimal extensions or simple extensions.

We begin with case (f).

\begin{proposition}\label{3.8} Let $(R,M)$ be a local ring and $R\subset S$ a   subintegral $M$-crucial extension. Then, $\ell[R,S]=2$ if and only if either $R\subset S$ is simple and  $|[R,S]|=3$ or $R\subset S$ is co-pointwise minimal.  In the last case, $|[R,S]|=|R/M|+3$.
\end{proposition}

\begin{proof} Clearly,  $S$ is a local ring.  Let $N$ be its maximal ideal.

Suppose that $\ell[R,S]=2$, then, $R\subset S$ is either simple or co-pointwise minimal  by  (Proposition~\ref{2.2}).  Assume first that $R\subset S$ is simple. Let $T\in]R,S[$ so that $R\subset T$  and  $T\subset S$ are minimal ramified  by  Table T$_2$. 
 Setting $M':=N\cap T$, which is the maximal ideal of $T$,  we get $M'=(T:S),\ N^2\subseteq M'$ and $T=R+M'$. Since $R\subset S$  is subintegral,  $R/M\cong S/N$ and $S=R+N$.  As $R\subset S$ is not  co-pointwise minimal, we deduce from \cite[Propositions 3.9 and  5.6 and Lemma 5.3]{Pic 7}, that either $(R:S)\neq M\ (1)$ or $N^2\not\subseteq M\ (2)$. In case (1), it follows that $MS\not\subseteq R$, but $MS\subseteq M'S=(T:S)S\subseteq T$ and then $T=R+MS$. In case (2), $N^2\not\subseteq R$, but $N^2\subseteq M'$ show that $R+N^2=T$. In any case, $T$ is uniquely defined by the properties of $M$ and $N$, so that $|[R,S]|=3$. 

Conversely, $|[R,S]|=3$ implies that $\ell[R,S]=2$. 

Assume that $R\subset S$ is  co-pointwise minimal. Then,  (Proposition~\ref{2.4}) yields that  $\ell[R,S]=2$. Under these conditions, $M=(R:S)$ and there is a poset isomorphism $[R,S]\cong [R/M,S/M]$, so that $|[R,S]|=|[R/M,S/M]|$. Set $k:=R/M$ and $S':=S/M$.  From (Proposition~\ref{2.4}), we infer that there are $x,y\in k\setminus \{0\}$ such that $S'=k+kx+ky$, with $x^2=y^2=xy=0$, so that $S'$  has $\{1,x,y\}$ as a basis  and is  a three-dimensional $k$-vector space. Let $T\in]k,S'[$ so that $T$ is a two-dimensional $k$-subalgebra of $S'$. Then, $T=k+kz=k[z]$, where $z=ax+by,\ (a,b)\in k^2\setminus\{(0,0)\}$. If $a\neq 0$, set $c:=a^{-1}b$ and $t_c:=a^{-1}z=x+cy$, so that $T=k[t_c]$. We get an injection $k\to ]k,S'[$ defined by $c\mapsto k[t_c]$. At last, if $a=0$, then $b\neq 0$ so that $k[by]=k[y]$ and  $k[y]\in]k,S'[$.  
 It follows that  $]k,S'[=\{k[t_c]\mid c\in k\}\cup\{k[y]\}$, giving $|]k,S'[|=|k|+1$, and  $|[R,S]|=|R/M|+3$.
\end{proof}

\begin{proposition}\label{3.81}  A   subintegral $M$-crucial extension $R\subset S$  has length 2 if and only if either $R\subset S$ is simple  and    $|[R,S]|=3$ or $R\subset S$ is co-pointwise minimal.  In the last case, $|[R,S]|=|R/M|+3$.  
\end{proposition}

\begin{proof} Use (Proposition~\ref{3.8}) and (Proposition~\ref{3.11}).
\end{proof}

We can say more for a simple extension of length 2.

\begin{corollary}\label{3.13} Let $(R,M)$ be a local ring and $R\subset S$ a $M$-crucial simple subintegral extension.  Then $S$ is  local and we set $\{N\}=\mathrm Max(S)$. 

(i )  
There is  some $y\in N$ such that $S=R[y]$.

(ii)  $\ell[R,S]=2$ if and only if $M^2\subseteq (R:S)\subseteq M$ and one of the following condition holds:
\begin{enumerate}
\item $(R:S)=M,\ N^2\not\subseteq M$ and $N^3\subseteq M$.
\item $(R:S)\neq M,\ y^2\not\in R,\ MS=M+N^2=M+Ry^2\subset N$ and $MN^2\subseteq M$.
\item $(R:S)\neq M,\ y^2\in R$ and $\dim_{R/M}((M+My)/M)=1$.
\end{enumerate} 

If these conditions  hold, then $[R,S]=\{R,R+N^2,S\}$. 
\end{corollary}

\begin{proof} By subintegrality,  $S$ is a local ring with maximal ideal $N$ and $S=R+N$. Since $R \subset S$  is simple, there is $z\in S$ such that $S=R[z]$. But $z=a+y$, for some $a\in R$ and $y\in N$, so that $R[z]=R[y]$. 

Assume first that $\ell[R,S]=2$, so that there is $T\in[R,S]$ such that $R\subset T$ and $T\subset S$ are minimal ramified and then  $T$ is a local ring with maximal ideal $M'$. From $MM'\subseteq M,\ M'^2\subseteq M\subset M',\ M'S=M'$ and $N^2\subseteq M'$, we deduce that $M^2S\subseteq MM'S= MM'\subseteq M$, so that $M^2\subseteq(R:S)\subseteq M$. In particular, $N^4\subseteq M$, which gives $S=R+Ry+Ry^2+Ry^3$ and $N=M+Ry+Ry^2+Ry^3$.

(1) If $(R:S)=M$, then $y^2\not\in M$, because,  if not, by Table T$_1$,  $R\subset S$ would be minimal ramified, whence $N^2\not\subseteq M$. Moreover, $M$ is an ideal of both  $R$, $T$ and $S$. Setting $k:=R/M,\ T':=T/M,\ S':=S/M,\ N':=N/M$ and  $y'$  the class of $y$ in $S'$, we get the tower $k\subset T'\subset S'=k[y']$, where $k$ is a field and $k\subset T'$ and $T'\subset S'$ are minimal ramified, so that there exists $x'\in T'$ such that $T'=k[x']$. Since $y'^2\in T'\setminus k$, there is no harm to choose $x'=y'^2$. Now,  $T'=k+kx'=k+ky'^2$ is a local ring with maximal ideal $kx'$ such that $x'^2=0$. Then, $N'=kx'+T'y'=kx'+ky'$ because $x'y'\in kx'$. We get $x'^3=x'^2y'=x'y'^2=x'^2=0$.  Finally,  from $y'^3=x'y'\in kx'$, we infer that there is some $a\in k$ such that $x'y'=ax'$ which implies that $x'(a-y')=0$ in $S'$. If $a\neq 0$, then $a-y'$ is a unit in $S'$, so that $x'=0$, a contradiction. Then, $a=0$ and $y'^3=0$. Therefore, $N'^3=0$ and then $N^3\subseteq M$. From $N^2\not\subseteq M$, we deduce  $R\subset R+N^2$. Now $N^2\subseteq M'$ shows that $y\not\in N^2$ and that $R\subset R+N^2\subseteq T$. Then, $T=R+N^2$ and $[R,S]=\{R,R+N^2,S\}$.

(2) Assume that $(R:S)\neq M$ and $y^2\not\in R$, so that $M\neq MS,N^2$. Since $T\subset S$ is minimal ramified,   $N^2\subseteq M'$. It follows that $T=R+MS=R+N^2$ and $M'=MS=M+N^2\subset N$. Finally, $MN^2\subseteq MM'\subseteq M$. Moreover, $y^2\in T$ implies $M'=M+Ry^2$. In particular, $[R,S]=\{R,R+N^2,S\}$.

(3) Assume that $(R:S)\neq M$ and $y^2\in R$. We show that $My\not\subseteq M$. Denying, we would have $Ry^2,\ Ry^3\subseteq M$, so that $(R:S)=M$, a contradiction. In particular, there is some $m\in M$ such that $my\not\in R$. Set $x_m:=my$. Then, $x_m^2=m^2y^2\in M$. Since $Mx_m=Mmy\subseteq M^2S\subseteq M$, we get that $R\subset R[x_m]$ is minimal ramified, so that $R[x_m]=T$ by (Proposition \ref{3.8}). In particular, $M'=M+Rm'y$ holds for any $m'\in M$ such that $m'y\not\in M$. It follows that $M'=M+My\subseteq M+N^2\subseteq M'$. Then, $T=R+MS=R+N^2$ because $T=R+M'\subseteq R+N^2\subseteq R+M'=T$. Since $MM'\subseteq M$ and $R\subset T$ is minimal ramified,  ${\mathrm L}_R(M'/M)=1=\dim_{R/M}((M+My)/M)$ holds. In particular, $[R,S]=\{R,R+N^2,S\}$.

Conversely, assume that $M^2\subseteq (R:S)\subseteq M$ and one of the  conditions (1), (2) or (3) holds.

(1) Assume that $(R:S)=M,\ N^2\not\subseteq M$ and $N^3\subseteq M$. We keep the same notation as in the direct part of (1) for $k,S',N'$ and $y'$. Then, $y'^3=0$ gives $S'=k+ky'+ky'^2$ and $N'=ky'+ky'^2$. It follows that $N'^2=ky'^2\neq 0$, so that $T':=k[y'^2]\neq k$. Set $x':=y'^2$. Then, $T'=k[x']$ verifies $x'^2=y'^4=0$, so that $k\subset T'$ is minimal ramified and $T'$ is a local ring with maximal ideal $kx'$. In particular, $T'\neq S'$ since $k\subset S'$ is not minimal. Finally, $y'^2=x'\in kx'$ and $x'y'=y'^3=0$ show that $T'\subset S'$ is minimal ramified. Then, $\ell[k,S']=2=\ell[R,S]$.

(2) Assume that $(R:S)\neq M,\ y^2\not\in R,\ MS=M+N^2=M+Ry^2\subset N$ and $MN^2\subseteq M$. In particular, $N^2\subseteq MS$ implies that $N^4\subseteq M^2S\subseteq M$. Set $x:=y^2\in N^2$. Then, $Mx\subseteq M$ and $x^2\in N^4\subseteq M$ show that $R\subset T:=R[x]$ is minimal ramified, with $T$ a local ring whose maximal ideal is $M':=M+Rx=M+Ry^2=MS=M+N^2$.  Moreover, $T\subset S=T[y]$ since $(R:S)\neq M=(R:T)$. Now, $y^2\in M'$ and $M'y=My+Ry^3$, with $My\subseteq MS=M'$ and $y^3\in N^2\subseteq M'$,  imply that $T\subset S$ is minimal ramified. 

(3) Assume that $(R:S)\neq M,\ y^2\in R$ and $\dim_{R/M}((M+My)/M)=1$. We get that $y^2\in M$ and $y^n\in M$ for any integer $n\geq 4$ since $M^2\subseteq (R:S)$. It follows that $S=R+Ry+Ry^3$. Set $M':=M+My$. Then, $M'$ is an ideal of $S$ containing strictly $M$, since $M'y=My+My^2\subseteq M'$ and $M'y^3\subseteq M'y\subseteq M'$. Set $T:=R+M'\in[R,S]\setminus \{R\}$. Then, $R\subseteq T$ is subintegral and  $MT\subseteq M+M^2+M^2y\subseteq M$, so that $M=(R:T)$. Now, $\dim_{R/M}((M+My)/M)=1={\mathrm L}_R(M'/M)$ shows that there exists $x\in M'\setminus M$ such that $M'=M+Rx$, giving $T=R+Rx=R[x]$.  Since $Mx\subseteq MM'=M^2+M^2y\subseteq M$ and $x^2\in M'^2\subseteq M^2+M^2y+M^2y^2\subseteq M$, we get that $R\subset T$ is minimal ramified.

Then  $T\neq S$ since $M=(R:T)\neq (R:S)$, $y^2\in M\subseteq M'$ and $M'y=My+My^2\subseteq My+M=M'$ show that $T\subset S$ is minimal ramified. 
\end{proof}

\begin{corollary}\label{3.132}  Let  $R\subset S$ be a simple subintegral   $M$-crucial extension.   Let $N$ be the only maximal ideal of $S$ lying over $M$. There exists $y\in N$ such that $S=R[y]$.

Then,  $\ell[R,S]=2$ if and only if $M^2\subseteq (R:S)\subseteq M$ and one of the following condition holds:
\begin{enumerate}
\item $(R:S)=M,\ N^2\not\subseteq M$ and $N^3\subseteq M$.
\item $(R:S)\neq M,\ y^2\not\in R,\ MS=M+N^2=M+Ry^2\subset N$ and $MN^2\subseteq M$.
\item $(R:S)\neq M,\ y^2\in R$ and $\dim_{R/M}((M+My)/M)=1$.
\end{enumerate} 

If these conditions  hold, then $[R,S]=\{R,R+N^2,S\}$. 
\end{corollary}

\begin{proof}  Use (Proposition~\ref{3.8}) and (Corollary~\ref{3.13}).
\end{proof}

\begin{example}\label{3.131} We are going to give  three examples satisfying each condition of (Corollary~\ref{3.13}) and one example showing that (Proposition~\ref{3.7}) does not hold when $(R:S)\neq M$ and $|[R,S]|=4$.

(1) Let $k$ be a field and set $S:=k[Y]/(Y^3)=k[y]$, where $y$ is the class of $Y$ in $S$. Set $x:=y^2\neq 0$, which satisfies $x^2=0$, so that $k\subset k[x]$ is minimal ramified, and  $T:=k[x]$ is a local ring with maximal ideal $M':=kx$. Moreover, $T\neq S$ because $y\not\in T$. Since $y^2=x\in M'$ and $xy=y^3=0\in M'$, we get that $T\subset S$ is also minimal ramified and we are in case (1) of (Corollary~\ref{3.13})

(2) Let $(R,M)$ be a SPIR such that $M=Rt\neq 0$, with $t^2=0$. Set $S:=R[Y]/(tY^2,tY-Y^2,Y^3)$ and let $y$ be the class of $Y$ in $S$. We have an extension $R\subset S$ and $(R:S)\neq M$ since $ty\not\in R$.  Set $x:=y^2\not\in R$ and $T:=R[x]$. Then, $x^2=tx=0\in M$ show that $R\subset T$ is  minimal ramified. Moreover, $T=R+Rx$ is a local ring with maximal ideal $M':=M+Rx=Rt+Rx$. We claim that $T\subset S$. Deny, and assume that $y\in T$, so that $y\in M'$, since $y^2\in M'$. There exist $a,b\in R$ such that $y=at+bx=at+by^2$, which gives $0\neq yt=at^2+bty^2=0$, a contradiction. To end, $ty=y^2=x\in M'$ and $xy=y^3=0\in M'$ show that $T\subset S$ is minimal ramified, $S=T+Ty=R+Rx+Ry$ is a local ring with maximal ideal $N:=M'+Ty=Rt+Rx+Ry$. Then  (Corollary~\ref{3.13}(2)) holds since $y^2\not\in R,\ MS=Rt+Rtx+Rty=Rt+Ry^2=M+Ry^2=M+N^2\subset N$ because $N^2=Ry^2$. Finally, $MN^2=0\subset M$. 
 
(3) Let $(R,M)$ be a SPIR such that $M=Rt\neq 0$, with $t^2=0$. Set $S:=R[Y]/(Y^2)$, which gives an extension $R\subset S$. Let $y$ be the class of $Y$ in $S$. Then, $(R:S)\neq M$ because $ty\not\in R$. Set $x:=ty\not\in R$ and  $T:=R[x]$. Then, $x^2=tx=0\in M$ show that $R\subset T$ is  minimal ramified. Moreover, $T=R+Rx$ is a local ring with maximal ideal $M':=M+Rx=Rt+Rx$. We claim that $T\subset S$. Deny, and assume that $y\in T$, so that $y\in M'$, since $y^2=0\in M'$. There exist $a,b\in R$ such that $y=at+bx=at+bty$, which gives $0\neq yt=at^2+bt^2y=0$, a contradiction. Now  $ty=x\in M'$ and $xy=ty^2=0\in M'$ show that $T\subset S$ is minimal ramified, $S=T+Ty=R+Ry$ is a local ring with maximal ideal $N:=M'+Ty=Rt+Ry$. Then (Corollary~\ref{3.13}(3)) holds  since $y^2=0\in R$. Moreover, $M+My=Rt+Rty=Rt+Rx=M'$  verifies  $\dim_{R/M}((M+My)/M)=\dim_{R/M}(M'/M)={\mathrm L}_R(M'/M)=1$ since $R\subset T$ is  minimal ramified  \cite[Lemma 5.4]{DPP2}.

 (4) Let $R$ be a DVD with maximal ideal $M=Rt$. Set $S:=R[Z]/(Z^2-Z,t^3Z)=R[z]$, where $z$ is the class of $z$ in $S$ and satisfies $z^2=z$ and $t^3z=0$. Set $x:=t^2z,\ y:=tz$, and $T:=R[x],\ T':=R[y]$. Since $x^2=tx=0$, we get by T$_1$ that $R\subset T$ is minimal ramified, so that $T$ is a local ring with maximal ideal $N:=M+Rx=Rt+Rx$. Now, $y^2=ty=x\in N$ and $yx=0$ show that $T\subset T'$ is minimal ramified, so that $T'$ is a local ring with maximal ideal $N':=N+Ty=Rt+Ry$. It follows that $R\subset T'$ is a subintegral simple extension of length 2 because $T'=R[y]$. Moreover, $M^2\subseteq (R:T')\subseteq M$. At last, $ty=x\not\in R$ shows that $(R:T')\neq M$ and $MT'=Rt+Rx=M+Ry^2=M+N'^2\subset N'$ and $MN'^2\subseteq M$. Then, (Corollary~\ref{3.13}(2)) shows that $[R,T']|=3$. 

Since $z^2-z=0\in N'$ and $tz=y\in N',\ ty=x\in N',\ tx=0\in N'$, we get that $T'\subset S$ is minimal decomposed by T$_1$, so that $T'={}_S^+T$. In particular, $T\subset S$ is an infra-integral $N$-crucial extension of length 2 with ${}_S^+T=T'\neq T,S$ by T$_2$. We also get that $R\subset S$ is  infra-integral with $T'={}_S^+R$.
 We claim that $|[R,S]|=4$. Deny, so that  there exists 
$T''\in]R,S[\setminus\{T,T'\}$. We may first  assume that $T''\subset S$ is minimal. Because of $T'\subset S$ minimal decomposed and ${}_S^+R=T'$, we have that $T''\subset S$ is ramified, since it cannot be decomposed.  Then, \cite[Lemma 17]{DPP4} yields that $MS$ is not a radical ideal of $S$ (only the FCP condition is necessary in the proof), a contradiction with $MS=Rt+Ry=N'=(T':S)$, the conductor of a decomposed  minimal extension. It follows that $R\subset T''$ is minimal with $T''\in [R,T']$, another contradiction since $|[R,T']|=3$. To conclude, $|[R,S]|=4$ with $(R:S)\neq M$, but $\ell[R,S]\neq 2$, showing that the condition $(R:S)=M$ is necessary to have $\ell[R,S]=2$ when $|[R,S]|=4$ and  $R\subset S$  infra-integral and $M$-crucial, with ${}_S^+R\neq R,S$.
\end{example}

  Consider now case (c) of (Definition~\ref{4.5}).

\begin{proposition}\label{3.14} Let $(R,M)$ be  a local ring and $R\subset S$  a seminormal infra-integral  $M$-crucial extension. The following conditions are equivalent:

\begin{enumerate}
\item $\ell[R,S]=2$.
\item $|\mathrm {Max}(S)|=3$.
\item $|[R,S]|=5$.
\end{enumerate} 
\end{proposition} 

\begin{proof} Since $R\subset S$ is seminormal,  $(R:S)$ is a radical ideal of $R$ and $S$  \cite[Lemma 4.8]{DPP2}. Then (Proposition~\ref{1.11}) entails that $(R:S)\in\mathrm{Max}(R)$ because $R\subset S$ is integral $M$-crucial, so that $(R:S)=M$. 

(1) $\Rightarrow$ (2) If $\ell[R,S]=2$,  there exists $T\in[R,S]$ such that $R\subset T$ and $T\subset S$ are minimal decomposed. It follows that $|\mathrm{Max}(S)|=3$. 

(2) $\Rightarrow$ (3) Setting $\mathrm {Max}(S):=\{M_1,M_2,M_3\}$, we have  $M=\cap_{i=1}^3M_i$, so that $S/M\cong\prod_{i=1}^3 S/M_i\cong (R/M)^3$. Then $|[R,S]|=|[R/M,S/M]|=B_3=5$  \cite[Proposition 4.16]{DPP3}, where $B_3$ is the 3rd Bell number.

(3) $\Rightarrow$ (1) If $|[R,S]|=5=B_3$, we get that  $R\subset S$ has FIP, so that $|\mathrm {Max}(S)|=3$, giving (1) \cite[Lemma 5.4]{DPP2} and \cite[Proposition 4.16]{DPP3}. 
\end{proof}

\begin{proposition}\label{3.141} Let $R\subset S$ be a seminormal infra-integral  $M$-crucial extension. Then the following conditions are equivalent:

\begin{enumerate}
\item $\ell[R,S]=2$.
\item $|\mathrm V(MS)|=3$.
\item $|[R,S]|=5$.
\end{enumerate} 
\end{proposition} 

\begin{proof}  Use (Proposition~\ref{3.14}) and (Proposition~\ref{3.11}).
\end{proof}

We may remark that the results hold as well for simple extensions as for co-pointwise extensions.

\begin{corollary}\label{3.142} Let $R\subset S$ be a seminormal infra-integral  $M$-crucial extension of length 2. Then $R\subset S$ is simple if and only if $|R/M|\neq 2$.
\end{corollary} 

\begin{proof}   (Proposition \ref{2.4})  excludes co-pointwise minimal extensions since $S/M\cong (R/M)^3$.
\end{proof}

It remains to consider  when $R\subset S$ is  t-closed,  integral and  $M$-crucial, {\it i.e.} the case (a)  of (Definition \ref{4.5}). Then $(R:S)=M$, by the same reasoning as in the previous proof since $R\subset S$ is seminormal and $M\in\mathrm {Max}(S)$  \cite[Lemme 3.10]{Pic 1}. Because of the bijection $[R,S]\to [R/M,S/M]$, where $R/M\subset S/M$ is a field extension, we can reduce  our study to field extensions of length 2, which is achieved in next section. 

\section{ Field extensions of length 2} We will call in this paper {\it radicial} any purely inseparable field extension. We recall that a minimal field extension is either separable or radicial (\cite[p. 371]{Pic}). We will use the separable closure of an algebraic field extension.

\subsection{Finite non separable field extensions}

 Let $k\subset L$ be a  finite radicial extension and  $p:=\mathrm{c}(k)\in\mathbb P$. Then $[L:k]=p^n$ for some positive integer $n$,  $\ell[k,L]=n$ and any maximal chain of subextensions of $k\subset L$ has length $n$. Moreover, $k\subset L$ is minimal if and only if $[L:k]=p$. For a positive integer $n$, a radicial extension $k\subset L$ of fields is said to have height $n$ if $x^{p^n}\in k$ for each $x\in L$, and there exists $y\in L$ such that $y^{p^{n-1}}\not\in k$, 
 \cite[Proposition 1, A V.23]{Bki A}.

\begin{proposition}\label{3.15} Let $k\subset L$ be a finite  radicial field extension and  $p:=\mathrm{c}(k) \in \mathbb P$. Then   $\ell[R,S]=2$ if and only if either $k\subset L$ is  simple  of height 2 or  co-pointwise minimal.

If these conditions hold, then  $|[k,L]|=3$ when $k\subset L$ is a simple and $|[k,L]|=\infty$ when $k\subset L$ is a co-pointwise  minimal extension.
\end{proposition} 

\begin{proof} 

  Assume that $\ell[R,S]=2$. If $L$ is simple, there exists some $y\in L$ such that $L=k[y]$, which satisfies $y^{p^n}\in k$ for a least  integer $n$. In particular, for any $z\in L$, we have $z^{p^n}\in k$, so that $n$ is the height of the extension $k\subset L$. Let $K\in[k,L]$ be such that $k\subset K$ and $K\subset L$ are minimal field extensions. Then, they are both radicial field extensions of degree $p$. In particular, $y^p\in K$ and $(y^p)^p=y^{p^2}\in k$. It follows that for any $z\in L$, we have $z^{p^2}\in k$, so that $k\subset L$  of height 2. 

If $k\subset L$ is not simple, then it is a  co-pointwise minimal extension.

 Conversely, assume that  $k\subset L$ is a simple extension of height 2 such that $L=k[y]$, then $y^{p^2}\in k$ and $y^p\not\in k$. Set $K:=k[y^p]$.  Then,  $k\subset K$ and $K\subset L$ are minimal field extensions of degree $p$.  
Assume that there exists $K'\in]k,L[$. Then, $y\not\in K'$. If $y^p\in K'$, we get that $K\subseteq K'\subset L$, so that $K'=K$. If $y^p\not\in K'$, then, $[K':L]=p^2$, giving $k=K'$, a contradiction. Then, $[k,L]=\{k,K,L\}$,Ê $|[k,L]|=3$ and $\ell[R,S]=2$.

If $k\subset L$ is a  co-pointwise  minimal extension, then $\ell[R,S]=2$.

  Since a co-pointwise minimal extension  $k\subset L$   is not  simple,   the Primitive Element Theorem asserts that $|[k,L]|=\infty$.  
\end{proof}

We do not give here a special example since such an extension is of the form $k\subset k[y]$, with $y^{p^2}\in k$, where $p:=\mathrm{c}(k) \in \mathbb P$. 

Let $k\subset L$ be a finite field extension and $L_s$ (resp. $L_r$) be the separable (resp. radicial) closure of $k$ in $L$.   

\begin{definition}\label{except}  \cite{G} A finite field extension $k\subset L$ is said to be {\it exceptional} if $k=L_r$ and $L_s\neq L$. 
\end{definition}

\begin{proposition}\label{3.16} Let $k\subset L$ be a finite  field extension, which is neither separable, nor radicial. Then  $\ell[k,L]=2$ if and only if  one of the following condition holds:
\begin{enumerate}
\item $k\subset L$ is  exceptional and $|[k,L]|=3$.
\item $k\subset L$ is not exceptional and $|[k,L]|=4$.
\end{enumerate} 

If the above conditions hold, then $k\subset L$ is a simple extension. 
\end{proposition} 

\begin{proof} Since $k\subset L$ is neither separable nor radicial, $L_s\neq k,L$. Moreover, $k\subset L$ is not co-pointwise minimal by (Proposition ~\ref{2.4}). 

Assume first that $\ell[k,L]=2$. It follows that $k\subset L_s$ and $L_s\subset L$ are minimal. Assume that there exists some $K\in]k,L[\setminus\{ L_s\}$  so that $k\subset K$ is minimal. Then, either $k\subset K$ is separable, or $k\subset K$ is radicial. If $k\subset K$ is separable, then, $K\subseteq L_s$, and we get a contradiction, so that $k\subset K$ is radicial and $K\subseteq L_r$. If $k\subset L$ is exceptional, we get again a contradiction and then $|[k,L]|=3$. If $k\subset L$ is not exceptional, then $K\in[k,L_r]$  and we have the tower $k\subset K\subseteq L_r\subset L$ since $k\subset L$ is not radicial.   
Then, 
 $K=L_r$,  
  $[k,L]=\{k,L_s,L_r,L\}$ and $|[k,L]|=4$.

Conversely, assume that either $k\subset L$ is exceptional and $|[k,L]|=3$ or $k\subset L$ is not exceptional and $|[k,L]|=4$. In the first case, we have obviously $\ell[R,S]=2$. In the second case, we have $L_r\in]k,L[\setminus\{ L_s\}$, so that $[k,L]=\{k,L_s,L_r,L\}$. As $L_s$ and $L_r$ are not comparable, it follows that $k\subset L_s,\ L_s\subset L,\ k\subset L_r$ and $L_r\subset L$ are all minimal extensions, so that $\ell[k,L]=2$.
\end{proof}

\begin{example} (1) \cite[Ex. 3, A V.144]{Bki A}. Let $F$ be a field with $p:=\mathrm{c}(k)>2$ and let $k:=F(X,Y)$  be the rational function field in two indeterminates  $X,Y$ over $F$. Set $L:=k[\alpha]$, where $\alpha$ is a zero of the polynomial $f(Z):=Z^{2p}+XZ^p+Y$. Then $k\subset L$ is exceptional and $[L:k[\alpha^p]]=p$. Set $K:=k[\alpha^p]$. It follows that $k\subset K$ is minimal separable and $K\subset L$ is minimal radicial, so that $L_s=K$. Moreover, $L_r=k$ since $k\subset L$ is exceptional. To end, $|[k,L]|=3$ because $[k,L]=\{k,K,L\}$, giving that $\ell[k,L]=2$.

(2) The following example, due to Morandi \cite[Example 4.18, p.46]{ M}, illustrates case (2). Let $k:=\mathbb{F}_2(X)$ be the rational function field in one indeterminate over $\mathbb{F}_2$ and set $L:=k[\alpha]$, where $\alpha^6=X$. Then, $L_s:=k[\alpha^2]$ is the separable closure of $k$ in $L$ and $L_r:=k[\alpha^3]$ is the radicial closure of $k$ in $L$. It follows that all the extensions $k\subset L_r,\ k\subset L_r,\ L_r\subset L$ and $L_s\subset L$ are minimal. Moreover, $k\subset L$ is not exceptional. At last, $|[k,L]|=4$ since $[k,L]=\{k,L_r,L_s,L\}$. Deny and assume that there exists some $K\in]k,L[\setminus\{L_r,L_s\}$. We cannot have $k\subset K$ minimal, because in this case it would be either radicial, or separable, giving $K=L_r$ or $L_s$, a contradiction. Then, either $L_r\subset K$ or $L_s\subset K$, a contradiction. Hence $|[k,L]|=4$ and $\ell[k,L]=2$.
\end{example}

 \subsection{Finite  separable field extension}
 
 The last case to consider is  a finite separable field  $k \subset L$ extension of length 2. We need some new concept that will allow us to characterize minimal separable field extensions, namely the family of generating  principal subfields of $L$ introduced by van Hoeij, Kl\"uners and Novocin in \cite{HKN}.  The set of monic polynomials of $k[X]$ is denoted by $k_u[X]$.

From now on, our riding hypotheses for the section will be: $L:=k[x]$ is a separable (FIP) field extension of $k$  with degree $n$ and $f(X)\in k_u[X]$ is the minimal polynomial of $x$ over $k$. If $g(X)\in L_u[X]$ divides $f(X)$, we denote by $K_g$ the $k$-subalgebra of $L$ generated by the coefficients of $g$.  For any $K\in[k,L]$, we denote by $f_K(X)\in K_u[X]$ the minimal polynomial of $x$ over $K$. The proof of the Primitive Element Theorem shows that $K=K_{f_K}\ (*)$. Of course, $f_K(X)$ divides $f (X)$ in $K[X]$ (and in $L[X]$). If $f(X):=(X-x)f_1(X)\cdots f_r(X)$ is the decomposition of $f(X)$ into irreducible factors of $L_u[X]$, we  set $\mathcal F:=\{f_1,\ldots,f_r\}$  because the $f_{\alpha}'$s are different by separability. There are ring morphisms $p_{\alpha}:k[X]/(f(X))\cong L\to L[X]/(f_{\alpha}(X))$ for $\alpha=1,\dots,r$. If $L_{\alpha}$ is the pullback field associated to $p_{\alpha}$ and $L\to L[X]/(f_{\alpha}(X))$, we get subextensions $k\subseteq L_{\alpha}$ of $k\subseteq L$ according to the following diagram: 

$$\begin{matrix}
  k & \longrightarrow & L_{\alpha}    & \rightarrow                            &       L           \\
 {} & \searrow{}         & \downarrow &        {}                                    & \downarrow \\
 {} &        {}                &   k[X]/(f(X))  & \underrightarrow{\ p_{\alpha}\ } & L[X]/(f_{\alpha}(X))
\end{matrix}$$

The $L_{\alpha}$s are called the {\it principal subfields} of $k\subset L$. As we will see later, it may be that $L_{\alpha}=L_{\beta}$ for some $\alpha\neq \beta$. To get rid of this situation, we define $\Phi:\mathcal F\to [k,L]$ by $\Phi(f_{\alpha})= L_{\alpha}$. If $t:=|\Phi(\mathcal F)|$, we set $\Phi(\mathcal F):=\{E_1,\ldots,E_t\}:=\mathcal E$ and $m_{\beta}:=f_{E_{\beta}}$ for $\beta\in\mathbb N_t$.

 For $K\in[k,L[$ , we set $I(K):=\{\alpha\in  \mathbb N_r \mid f_K(X)/ f_{\alpha}(X) \in L_u[X] \}$ and we say that $I(K)$ is the subset of $\mathbb N_r$ {\it associated to} $f_K$. In particular, $I(k)=\mathbb N_r$. We also set $J(K):=\{\beta\in  \mathbb N_t\mid \Phi(f_{\alpha})=E_{\beta}$ for all $\alpha\in I(K)\}$.  For $\beta\in \mathbb N_t$, we define $\Gamma(\beta)$ as the set of $\alpha$ such that the $f_\alpha$ are in the same class of equivalence for the equivalence relation associated to $\Phi$, that is $\Gamma(\beta):=\{\alpha\in\mathbb N_r\mid E_{\beta}=\Phi(f_\alpha)\}$. Each $K\in[k\, L]$ is an intersection of some of the $E_{\beta}$s \cite[Theorem 1]{HKN} and more precisely, the proof of this theorem gives the next result.
 
\begin{theorem}\label{4.210} Let $K\in[k,L[$. Then,  $f_K(X)=(X-x)\prod_{\alpha\in I(K)}f_{\alpha}(X)$ and $K=\{g(x)\in L\mid g(X)\in k[X], g(X)\equiv g(x)(f_K(X))\}=\cap_{\beta\in J(K)}E_{\beta}$.
\end{theorem} 

It follows that $E_{\beta}=\{g(x)\in L\mid g(X)\in k[X],\  g(X)\equiv g(x)\ (f_{\alpha}(X))$ for any $\alpha$ such that $\Phi(f_{\alpha})=E_{\beta}\}.$ In the following, we write $K_{\alpha}:=K_{g_{\alpha}}$, where $g_{\alpha}(X):=(X-x)f_{\alpha}(X)$. In fact, $k=\cap_{\alpha=1}^rL_{\alpha}=\cap_{\beta=1}^tE_{\beta}.$

\begin{remark}  Let $R\subset S$ be a  t-closed $M$-crucial extension such that $R/M\subset S/M$ is a separable field extension. Because of the bijection $[R,S]\to [R/M,S/M]$ defined by $T\mapsto T/M$, there  exists a finite family $\{T_{\alpha}\}\subset [R,S[$ such that each element of $[R,S[$ is an intersection of some of the $T_{\alpha}$'s by (Theorem~\ref{4.210}).
\end{remark}

\begin{lemma}\label{4.22} Let $g(X)\in L_u[X]$ dividing $f(X)$ and such that  $g(X):=(X-x)g'(X)$ with $g'(X)$ irreducible and $K_g\neq L$, then $K_g\subset L$ is  minimal and   $g=f_{K_g}$. 
\end{lemma}

\begin{proof} Since $g'(X)\in L_u[X]$ is irreducible, divides $f(X)$ and is such that $g'(x)\neq 0$, it follows that $g'=f_{\alpha}$ for some ${\alpha}\in\{1,\ldots,r\}$. Assume that $K_g\neq L$, and let $K\in[K_g,L]$ be such that $K\subset L$ is minimal, so that $g(X)\in K[X]$. Then, $f_K(X)$ divides $g(X)$ in $K[X]$ since $g(x)=0$. Moreover, $f_K(X)=(X-x)h(X)$, with $h(X)\in L[X]\setminus L$. Then, $h$ divides $g'$ in $L[X]$, which is irreducible in $L[X]$, so that $h=g'$, giving $f_K=g$, whence $K=K_g$. Hence, $K_g\subset L$ is minimal and $g=f_{K_g}$. 
\end{proof}

 \begin{proposition}\label{4.221}  The following statements hold:
 \begin{enumerate}
\item $f_{\alpha}$ divides $f_{\Phi(f_{\alpha})}$ in $L[X]$ for each ${\alpha}\in \mathbb N_r$.
\item Let $\beta\in\mathbb N_t$. Then, $\Gamma(\beta)\subseteq I(E_{\beta})$ and (3) holds.
\item  $m_{\beta}(X)=(X-x)\prod [f_{\alpha}(X) \mid {\alpha}\in \Gamma(\delta),\ E_{\beta}\subseteq E_{\delta}]$.
\end{enumerate}
\end{proposition}

\begin{proof} (1) is \cite [Lemma 45]{SVH}.  (We benefited from a preprint and correspondence with M. van Hoeij).

(2) Let $\alpha\in \Gamma[\beta]$, so that $E_{\beta}=\Phi(f_{\alpha})$. In particular,  $f_{\alpha}$ divides $f_{E_{\beta}}=m_{\beta}$. It follows that ${\alpha}\in I(E_{\beta})$.

(3) Let $\beta\in\mathbb N_t$ and let $E_{\delta}$ be such that $E_{\beta}\subseteq E_{\delta}$. Then, $m_{\delta}$ divides $m_{\beta}$. Let $\alpha\in \Gamma(\delta)$. By (2), we get that $f_{\alpha}$ divides $m_{\delta}$, yielding $f_{\alpha}$ divides $m_{\beta}$. Conversely, let $f_{\alpha}$ dividing $m_{\beta}$ and set $E_{\delta}:=\Phi(f_{\alpha})$.  This implies that $E_{\beta}\subseteq E_{\delta}$ by (Theorem~\ref{4.210}). 
\end{proof}

\begin{remark}\label{R} The inclusion (2) of the previous proposition may be strict (see (Example~\ref{3.171})).
\end{remark}

\begin{lemma}\label{4.23}  The following statements hold:
\begin{enumerate}
\item $E_{\beta}\neq L$ for each $\beta\in \mathbb N_t$.
\item If $K_{\alpha}\neq L$, then $K_{\alpha}=\Phi(f_{\alpha})$. In particular, $K_{\alpha}\subset L$ is minimal for each $\alpha$ such that $K_{\alpha}\neq L$ and then $f_{K_{\alpha}}(X)=(X-x)f_{\alpha}(X)$.
\item Let $K\in[k,L]$  be such that $K\subset L$ is minimal. There exists some $\beta\in \mathbb N_t$   such that $K=E_{\beta}$.
\end{enumerate}
\end{lemma}

\begin{proof}  (1) Assume that $E_{\beta}=L$ for some $\beta$, so that $x\in E_{\beta}$. Let $\alpha\in\Gamma(\beta)$ and  set $g(X):=X$. Since $g(x)=x\in E_{\beta}$, the characterization of $E_{\beta}$ entails that $f_{\alpha}(X)$ divides $g (X)-g(x)=X-x$, a contradiction, since $f(X)$ is separable. Therefore, $E_{\beta}\neq L$. 

(2) Assume that $K_{\alpha}\neq L$. By (Lemma~\ref{4.22}),  $K_{\alpha}\subset L$ is  minimal  and $g_{\alpha}=f_{K_{\alpha}}$. In view of (Theorem \ref{4.210}), we get that $K_{\alpha}=\cap_{{\beta}\in J(K_{\alpha})}E_{\beta}$, where $f_{K_{\alpha}}(X)=(X-x)\prod_{{\delta}\in I(K_{\alpha})}f_{\delta}(X)=g_{\alpha}(X)=(X-x) f_{\alpha}(X)$, so that $I(K_{\alpha})=\{\alpha\}$.  Then, $\Phi(f_{\alpha})=K_{\alpha}=E_{\beta}$ such that $\Gamma(\beta)=\{\alpha\}$.
\end{proof}

As a by-product, using the pullbacks $E_j$, we get a characterization of minimal separable fields extensions. Our result is completely independent of  Galois theory, contrary to  Philippe's  methods  \cite{Ph}.   She proved that  a separable extension $k\subset k(x)$ is minimal if and only if the Galois group of the minimal polynomial of $x$ is primitive \cite[Proposition 2.2(3)]{Pic}.

\begin{proposition}\label{4.252} Let $k\subset L=k[x]$ be a finite separable  field extension of degree $n$. Then  the following statements	are equivalent:
\begin{enumerate}
\item $k\subset L$ is a minimal extension.
\item $t=1$.
\item For each $\alpha\in\mathbb N_r$ and each $h(X)\in k[X]\setminus k$ with degree $<n,\ h(X)\not\equiv h(x)\ (f_{\alpha}(X))$.
\item $K_g=L$ for each $g(X)\in L[X]$ dividing strictly $f(X)$ and such that $g(x)=0$.
\end{enumerate}
In particular, if $r=1$, then $k\subset L$ is  minimal.
\end{proposition} 

\begin{proof} (1) $\Rightarrow$ (2)  From (Lemma \ref{4.23}), we deduce that $E_{\beta}\neq L$ for each $\beta$, so that $E_{\beta}=k$ and then $t=1$.

(2) $\Rightarrow$ (3) Since $t=1$, we get that $E_1=k$ is the only $E_{\beta}$. Assume that there is some $h(X)\in k[X]\setminus k$ with degree $<n$ such that $h(X)\equiv h(x)\ (f_{\alpha}(X))$ and set $K:=k[h(x)]$. Then, $h(X)\equiv h(x)\ (f_K(X))$, so that $\deg(f_K)<n$, whence $K\neq k$. By (Theorem \ref{4.210})  $K$ is an intersection of some $E_{\beta}$'s, a contradiction. Then, for each $h(X)\in k[X]\setminus k$ with degree $<n,\ h(X)\not\equiv h(x)\  (f_{\alpha}(X))$, for each $\alpha\in \mathbb N_r$. 

(3) $\Rightarrow$ (1) Assume that $k\subset L$ is not minimal and let $K\in [k,L]\setminus\{k,L\}$. Let $h(X)\in k[X]\setminus k$ with degree $<n$ such that $K=k[h(x)]$. In particular, $\deg(h)>1$. By (Theorem \ref{4.210}),  some $f_{\alpha}(X)$ divides both $f_K(X)$ and $h(X)-h(x)$, a contradiction. Then, $k\subset L$ is  minimal.

(1) $\Leftrightarrow$ (4) $k\subset L$ is not  minimal if and only if there exists some $K\in [k,L]$ such that $k\subset K\subset L$. In this case,  $f_K(X)$ divides strictly $f(X)$ and $f_K(x)=0$ with $K=K_{f_K}\neq L$. Conversely, assume that there is some $g(X)\in L[X]$ dividing strictly $f(X)$ and such that $g(x) =0$ with $K_g\neq L$. Then, $k\subset K_g\subset L$ follows and $k\subset L$ is not  minimal.

If, in particular, $r=1$, then, $t=1$ and  $k\subset L$ is  minimal.
\end{proof}
\begin{example} \label{4.2521} If $k\subset L$ is  Galois of degree 3, then $f(X)=(X-x)(X-x_1)(X-x_2)$, with $x,x_1,x_2\in L\setminus k$ all distincts. Set $f_{\alpha}(X)=X-x_{\alpha}$, for $\alpha=1,2$ and $g_i(X)=(X-x)f_{\alpha}(X)$. Then, $K_{\alpha}=k[x+x_{\alpha},xx_{\alpha}]=k[x_{\beta}]=L$, where $\alpha\neq \beta\in\{1,2\}$. So, we recover the fact that $k\subset L$ is minimal, giving $t=1$ and $E_1=k$

\end{example}

For each $\beta\in\mathbb N_t$, we set $\mathcal{F}_{\beta}:=\{f_{\alpha}\in \mathcal{F}\mid f_{\alpha}$ divides $m_{\beta}\}$, so that $m_{\beta}(X)=(X-x)\prod_{f_{\alpha}\in\mathcal{F}_{\beta}}f_{\alpha}(X)$.

\begin{proposition}\label{4.253} The following conditions are equivalent:
\begin{enumerate}

\item $E_{\beta}\subset L$ is minimal for each $\beta\in\mathbb N_t$.  
\item For each $\beta\in\mathbb N_t,\ \mathcal{F}_{\beta}=\{f_{\alpha}\in \mathcal{F}\mid \Phi(f_{\alpha})=E_{\beta}\}$.  
\end{enumerate}
If these conditions hold, $\ell[k,S]=2$ if and only if $|[k,L]|=t+2$.
\end{proposition} 

\begin{proof} (1) $\Rightarrow $(2) Let $\beta\in\mathbb N_t$. Then, $m_{\beta}(X)=(X-x)\prod_{f_{\alpha}\in\mathcal{F}_{\beta}}f_{\alpha}(X)$. Let $f_{\alpha}\in \mathcal{F}_{\beta}$ and set $E_{\delta}:=\Phi(f_{\alpha})$.  By Theorem \ref{4.210},  $E_{\beta}\subseteq E_{\delta}\subset L$ which implies $E_{\beta}=E_{\delta}$ since $E_{\beta}\subset L$ is minimal, so that  $\Phi(f_{\alpha})=E_{\beta}$. Conversely, let $\alpha\in\mathbb N_r$ be such that $\Phi(f_{\alpha})=E_{\beta}$, whence $\alpha\in\Gamma(\beta)\subseteq I(E_{\beta})$ by (Proposition~\ref{4.221}). Therefore $f_{\alpha}$ divides $m_{\beta}$, giving $f_{\alpha}\in \mathcal{F}_{\beta}$. 

(2) $\Rightarrow $(1)  Assume that $E_{\beta}\subset L$ is not minimal. Then, there is some $K\in]E_{\beta},L[$. By (Theorem \ref{4.210}), $K$ is an intersection of some $E_{\delta}$, for some $\delta\in\mathbb N_t$. In particular, $E_{\beta}\subset K\subseteq E_{\delta}$, giving $m_{\delta}$  divides strictly $m_{\beta}$. Let $\alpha\in I(E_{\delta})$, so that $f_{\alpha}$ divides $m_{\delta}\ (*)$, and also $m_{\beta}\ (**)$. But $(*)$ implies that $f_{\alpha}\in\mathcal{F}_{\delta}$, giving $\Phi(f_{\alpha})=E_{\delta}$ and $(**)$ implies that $f_{\alpha}\in\mathcal{F}_{\beta}$, giving $\Phi(f_{\alpha})=E_{\beta}$. To conclude, we get $E_{\delta}=E_{\beta}$, a contradiction. Then, $E_{\beta}\subset L$ is  minimal.

Assume that these conditions hold. If $k=E_{\beta}$ for some $\beta$, we get that $k\subset L$ is minimal, $t=1,\ \ell[k,L]=1$ and $|[k,L]|=2$. 

Now, assume that  $k\neq E_{\beta}$ for any $\beta$. 

If $\ell[k,L]=2$, let $K\in]k,L[$, so that $k\subset K$ and $K\subset L$ are both minimal. But, $K=\cap_{\beta\in J(K)}E_{\beta}\subset L$. It follows that $K=E_{\beta}$ for one $\beta\in J(K)$, since $E_{\beta}\neq L$, so that $[k,L]=\{k,E_1,\ldots,E_t,L\}$ giving $|[k,L]|=t+2$. Conversely, if $|[k,L]|=t+2$, we get that $[k,L]=\{k,E_1,\ldots,E_t,L\}$ since $E_{\beta}\neq E_{\delta}$ for $\beta,\delta\in\{1,\ldots,t\},\ \delta\neq\beta$. But $E_\beta\subset L$ being minimal for each $\beta\in\{1,\ldots,t\}$, the $E_\beta$'s are incomparable and we get that $\ell[k,L]=2$.
\end{proof}

\begin{theorem}\label{3.17} The  two following statements are equivalent:

\begin{enumerate}

\item   $\ell[k,L]=2$ 

\item  $t>1$ and $E_{\alpha}\cap E_{\beta}=k$ for all $\alpha,\beta\in\mathbb N_t,\ {\alpha}\neq \beta$. 
\end{enumerate}

If the above  statements hold, then

\begin{enumerate}

\item[(3)] $|[k,L]|=t+1$ if $k$ is one of the $E_{\alpha}$  

\item[(4)] $|[k,L]|=t+2$ if $k\neq E_{\alpha}$ for all $\alpha\in\mathbb N_t$. 
\end{enumerate}

\end{theorem} 

\begin{proof}  If (1) holds, then, $t>1$. Deny. Then $t=1$. In this case, $k=E_1$ by (Theorem~\ref{4.210}) and $]k,L[=\emptyset$, so that $k\subset L$ is minimal, a contradiction. Let $\alpha,\beta\in\mathbb N_t,\ \alpha\neq\beta$. It follows that  $k\subseteq E_{\alpha}\cap E_{\beta}\subseteq E_{\alpha},E_{\beta}\subset L$, with either $E_{\alpha}\cap E_{\beta}\subset E_{\alpha}$ or $E_{\alpha}\cap E_{\beta}\subset E_{\beta}$. In any case, $E_{\alpha}\cap E_{\beta}=k$ for all $\alpha,\beta\in\mathbb N_t,\ \alpha\neq \beta$. Hence, (2) is proved.

 Assume that (2) is valid. 
 Since any $K\in[k,L[$ is equal  to an intersection of some $E_{\alpha}$, we get that $[k,L[=\{k,E_{\alpha},\ \alpha\in\mathbb N_t\}$. We claim that two $E_{\alpha},E_{\beta}\neq k$ are incomparable. If not, $E_{\alpha}\subset E_{\beta}$ for some $\alpha,\beta\in\mathbb N_t,\ \alpha\neq \beta$ and then $k=E_{\alpha}=E_{\alpha}\cap E_{\beta}$, a contradiction. Therefore, (1) holds.
\end{proof}

\begin{example} \label{3.171} Set $k:=\mathbb{Q}$ and $L:=k[x]\subset \mathbb{R}$, where $x:=\root 4\of{2}$ is the positive real zero of the polynomial $f(X):=X^4-2$, the minimal polynomial of $x$ over $k$. Then, $f(X)=(X-x)(X+x)(X^2+x^2)$, a product of irreducible polynomials in $L[X]$. Set $f_1(X)=X+x$ and $f_2(X)=X^2+x^2$. Then, $g_1(X)=X^2-x^2$ and $g_2(X)=X^3-xX^2+x^2X-x^3$, giving $K_1=L_1=k[x^2]=E_1$ and $K_2=L$. In order to determine all elements of $[k,L]$, it remains to calculate $E_2=L_2=\{g(x)\in L\mid g(X)\in k[X],\ g(X)\equiv g(x)\ (f_2(X))\}$. In fact, it is enough to consider $g(X)\in k[X]$ of degree $< 4$. Set $g(X)=aX^3+bX^2+cX+d,\ a,b,c,d\in k$. Then, $g(X)-g(x)=(X-x)[a(X^2+xX+x^2)+b(X+x)+c]\ (*)$. It follows that $f_2(X)$ divides $g(X)-g(x)\Leftrightarrow X^2+x^2$ divides $a(X^2+xX+x^2)+b(X+x)+c\Leftrightarrow a(-x^2+ix^2+x^2)+b(ix+x)+c=0\Leftrightarrow a=b=c=0$. So, $g(x)\in E_2\Leftrightarrow g(x)=d$, for any $d\in k$, giving $E_2=k$. Then, $[k,L]=\{k=E_2,K_1=E_1,L=K_2\},\ |[k,L]|=3$ and $\ell[k,L]=2$. In particular, $m_1=g_1,\ m_2=f$  and $\Gamma(2)=\{2\}\subset I(E_2)=\{1,2\}$ (see notation before Theorem ~\ref{4.210} and Remark ~\ref{R}). 

\end{example} 

The case of a Galois extension is   a lot simpler. 

\begin{proposition}\label{3.18} Suppose that $k\subset L$ is a finite Galois   extension. 
\begin{enumerate}
\item If $[L:k]$ is a product of two prime integers, then $\ell[k,L]=2$.
\item Assume that $\ell[k,L]=2$ and $[L:k]=n$. Then $|[k,L]|\leq n+1$. 
\item Let $k\subset L$ be an Abelian field extension. Then $\ell[k,L]=2$ if and only if $[L:k]$ is a product of two prime integers $p$ and $q$. In this case, $|[k,L]|=3$ if $p=q$ and $|[k,L]|=4$ if $p\neq q$.
\end{enumerate}
\end{proposition} 

\begin{proof} (1) Assume that $[L:k]=pq$, with $p$ and $q$ two prime integers. Then, $k\subset L$ is not minimal \cite[Proposition 2.2 (2)]{Pic}. Let $k=K_0\subset K_1\subset\ldots\subset K_i\subset\ldots\subset K_n=L$ be a maximal chain of intermediate fields.  Then, $[L:k]=\prod_{i=0}^{n-1}[K_{i+1}:K_i]$ implies $n=2=\ell[k,L]$. 

(2) Assume that $\ell[k,L]=2$ and $[L:k]=n$. Then, $f(X)=(X-\alpha)\prod_{i=1}^{n-1}(X-\alpha_i)\in L(X)$. The number of the different $L_i$'s is at most $n-1$, so that $|[k,L]|\leq n+1$ by (Theorem~\ref{3.17}).

(3) If $k\subset L$ is Abelian, it is Galois. 
Assume  that $\ell[k,L]=2$ and let $K\in]k,L[$. Then, $K\subset L$ is minimal and Galois, so that $[L:K]=p$, a prime integer. Moreover, $k\subset K$ is also minimal Galois (see the Fundamental Theorem of Galois Theory \cite[Theorem 5.1, p. 51]{M}). It follows that $[K:k]=q$, a prime integer, so that $[L:k]=pq$. The converse is (1).

Now, if these conditions hold, $|[k,L]|$ is the number of subgroups of the Abelian group Gal$(L/k)$, equal to 3 if $p=q$ and 4 if $p\neq q$ (see \cite[Th\' eor\`eme 3 and Corollaire 2, AVII, p.22]{Bki A}) because Gal$(L/k)\cong\mathbb{Z}/p^2\mathbb{Z}$ or $\mathbb{Z}/pq\mathbb{Z}$.
\end{proof} 

We end by two  examples giving explicitly the subfields $L_{\alpha}$ and using results of \cite[Examples 5.2, p.52 and 5.3, p.53]{M}. We  give only an outline of the proofs. The method is the same as in (Example~\ref{3.171}).

\begin{example} \label{3.199} (1) Set $k:=\mathbb{Q},\ y:=\root 3\of{2}$ and $L:=k[x]$, where $x:=(1-j)y$ with $j=(-1+i\sqrt 3)/2$, a third root of unity. Then, $k\subset L$ is a Galois extension of degree 6, $\ell[k,L]=2$   
 and  $f(X)=X^6+108=X^6-x^6 = (X-x) f_1(X)f_2(X)f_3(X)f_4(x)f_5(X) $
 where $f_1(X)=X+x,\ f_2(X)=X-jx,\ f_3(X)=X+j^2x,\ f_4(X)=X+jx,\ f_5(X)=X-j^2x$.  We get:

 $g_1(X)=(X-x)(X+x)$ and $K_1=L_1=E_1=k[j^2y]$.

$g_2(X)=(X-x)(X-jx)$ and $K_2=L$ and $L_2=E_2=k[j]$.

$g_3(X)=(X-x)(X+j^2x)$ and $K_3=k[y]\neq L$, so that $K_3=L_3=E_3$.

$g_4(X)=(X-x)(X+jx)$, $K_4=k[jy]\neq L$, whence $K_4=L_4=E_4$.

$g_5(X)=(X-x)(X-j^2x)$, $K_5=k[jy,y]=L$, $L_5=k[j]=L_2=E_2$.
Therefore  $[k,L]=\{k,E_1,E_2,E_3,E_4,L]$ and $|[k,L]|=6$. In particular, $m_1=g_1,\ m_2=f_2g_5=f_5g_2,\ m_3=g_3,\ m_4=g_4 $. 

(2) Set $k:=\mathbb{Q}$ and $L:=k[x]$, where $x:=\sqrt 3+\sqrt 2$. Then, $k\subset L$ is a Galois extension of degree 4, $\ell[k,L]=2$ and  $f(X)=X^4-10X^2+1=(X-x)f_1(X)f_2(X)f_3(X)$,  where $f_1(X)=X+x,\ f_2(X)=X-x^{-1},\ f_3(X)=X+x^{-1}$.  

Set   $k_i:=k[\sqrt i],\ i=2,3,6$. Then, $g_1(X)=(X-x)(X+x)$, 
$g_2(X)=(X-x)(X-x^{-1})$ and 
$g_3(X)=(X-x)(X+x^{-1})$ giving 
 $L_1=K_1=E_1=k_6,\ L_2=K_2=E_2=k_3,\ L_3=K_3=E_3=k_2$ by 
 (Lemma~\ref{4.23}). 

To conclude, we have $[k,L]=\{k,E_1,E_2,E_3,L\},\ |[k,L]|=5$ and  the following diagram:
$\begin{matrix}
   {}  &        {}      & L             &       {}       & {}     \\
   {}  & \nearrow & \uparrow  & \nwarrow & {}     \\
E_1 &       {}       & E_2         &      {}        & E_3 \\
  {}   & \nwarrow & \uparrow & \nearrow & {}     \\
  {}   &      {}        & k             & {}             & {} 
\end{matrix}$

 Contrary to (1), we have $E_i=K_i$ for all $i=1,2,3$.  
\end{example}

 \section{Summing up     length 2 extensions characterization}
 We are now able to sum up the results of Sections 3,4 and 5 with respect to the cardinality of $[R,S]$ for an  extension of length 2.
 
\begin{theorem} \label{3.19} A ring extension $R\subset S$ is   of length 2 if and only if one of the following conditions hold: 
\begin{enumerate}

\item  $|\mathrm{Supp}(S/R)|=2,\ \mathrm{Supp}(S/R)\subseteq \mathrm{Max}(R)$ and $|[R,S]|=4$.

\item  $|\mathrm{Supp}(S/R)|=2,\ \mathrm{Supp}(S/R)\not\subseteq \mathrm{Max}(R)$ and $|[R,S]|=3$. Such an extension is  Pr\"ufer.

\item  $R\subset S$ is a non-integral $M$-crucial extension and  $|[R,S]|=3$. Such an extension  satisfies  $\overline R\neq R,S$.

\item $R\subset S$ is an integral $M$-crucial extension such that ${}_S^tR\neq R,S$ and $|[R,S]|=3$. 

\item $R\subset S$ an infra-integral $M$-crucial extension such that ${}_S^+R\neq R,S$ and either  $|[R,S]|=3$ or $(R:S)= M$ with $|[R,S]|=4$. 

\item $R\subset S$ is a subintegral $M$-crucial extension with either $|[R,S]|=3$  or $R\subset S$ is co-pointwise minimal. In this last case, $|[R,S]|=|R/M]+3$.

\item $R\subset S$ is a seminormal infra-integral $M$-crucial extension such that  $|[R,S]|=5$.

\item  $R\subset S$ is a  t-closed integral $M$-crucial extension, so that $M=(R:S)$, and the field extension $R/M\subset S/M$ satisfies one of the following conditions:

(a)  $R/M\subset S/M$ is  radicial and either $|[R,S]|=3$   or  $|[R,S]|=\infty$ with $R\subset S$  co-pointwise minimal. 

 (b)  $R/M\subset S/M$ is neither radicial nor separable, but  exceptional and  $|[R,S]|=3$. 

 (c) $R/M\subset S/M$ is neither radicial nor separable, nor exceptional and $|[R,S]|=4$.

 (d) $R/M\subset S/M$ is a finite separable field extension and $|[R,S]|=t+2$, where $t$ is the number of principal subfields of $S/M$ different from $R/M$.
\end{enumerate} 
\end{theorem}

\begin{proof}  Propositions \ref{3.1}, \ref{3.2}, \ref{3.3}, \ref{3.6}, \ref{3.7},  \ref{3.81},  \ref{3.141},  \ref{3.15}, \ref{3.16}, and Theorem  \ref{3.17}.
\end{proof} 

 \begin{remark}\label{7} (1) With the notation of (Theorem~\ref{3.19}), $R\subset S$ is  simple in cases (1)-(5), (8) (b)-(d).
  
(2) In case (8)(d),  $R\subset S$ is t-closed of length 2 with $M=(R:S)$. Let $x\in S$ be such that $S=R[x]$ and let $\overline x$ be the class of $x$ in $S/M$. Let $P(X)=X^n+\sum_{i=0}^{n-1}a_iX^{i}\in R[X]$ be a monic polynomial of least degree such that $P(x)=0$. Then, $\overline P(X):=X^n+\sum_{i=0}^{n-1}\overline a_iX^{i}\in(R/M)[X]$ is a monic polynomial such that $\overline P(\overline x)=0$. Let $\overline f(X)=X^l+\sum_{i=0}^{l-1}\overline b_iX^{i}\in (R/M)[X]$ be the minimal polynomial of $\overline x$. Then, $\overline f$ divides $\overline P$ in $(R/M)[X]$ and $l\leq n$. But $f(\overline x)=\overline x^l+\sum_{i=0}^{l-1}\overline b_i\overline x^{i}=0\in R/M$ gives that $x^l+\sum_{i=0}^{l-1}b_ix^{i}=m\in M$, so that $f(X):=X^l+\sum_{i=0}^{l-1}b_iX^{i}-m\in R[X]$ is a monic polynomial such that $f(x)=0$. Then, $n=l$ by the choice of $n$. It follows that $|[R,S]|=t+2\leq n+1$, where $n=[S/M:R/M]$.

   \end{remark}

  We recall below \cite[Theorem 4.1]{DS} and give a  table showing the link between (Theorem \ref{3.19}) and \cite[Theorem 4.1]{DS}. 
 
 \begin{theorem} \label{6.11} Let $R\subset S$ and $S\subset T$ be minimal extensions, whose crucial maximal ideals are respectively $M$ and $N$. Then $R\subset T$ has FIP if and only if (exactly) one of the following conditions holds:
 
 (i) Both $R\subset S$ and $S\subset T$ are integrally closed.
 
 (ii) $R\subset S$ is integral  and $S\subset T$ is integrally closed.
 
 (iii) $R\subset S$ is  integrally closed,  $S\subset T$ is integral, and $N\cap R\not\subseteq M$.
 
 (iv) Both $R\subset S$ and $S\subset T$ are integral, and $N\cap R\not\subseteq M$.

(v)  Both $R\subset S$ and $S\subset T$ are inert, $N\cap R= M$, and either $R/M$ is finite or there exists $\gamma\in T_M$ such that $T_M=R_M[\gamma]$. 

(vi) $R\subset S$ is decomposed, $S\subset T$ is inert and $N\cap R= M$.

(vii) Both $R\subset S$ and $S\subset T$ are decomposed and $N\cap R= M$.

(viii) $R\subset S$ is  inert, $S\subset T$ is decomposed and $N\cap R= M$.

(ix) $R\subset S$ is  ramified, $S\subset T$ is decomposed and $N\cap R= M$.

(x) $R\subset S$ is  decomposed, $S\subset T$ is ramified and $N\cap R= M$.

(xi) $R\subset S$ is  ramified, $S\subset T$ is inert and $N\cap R= M$.

(xii) $R\subset S$ is  inert, $S\subset T$ is ramified,  $N\cap R= M$, and the two conditions stated in \cite[Proposition 3.5 (a)]{DS} hold.

(xiii) Both $R\subset S$ and $S\subset T$ are ramified,  $N\cap R= M$, and the two conditions stated in \cite[Proposition 3.5 (b)]{DS} hold.
\end{theorem}

 There is not a one-to-one correspondence between  the different statements of Theorem \ref{3.19} and \cite[Theorem 4.1]{DS}. Obviously, any extension of length 2 (in Theorem \ref{3.19}) is of the form of some extension of \cite[Theorem 4.1]{DS}, but it has not necessarily FIP (for example if $R\subset S$ is a co-pointwise minimal t-closed extension (Proposition \ref{3.15}), or if $R\subset S$ is a co-pointwise minimal subintegral $M$-crucial extension such that $|R/M|=\infty$ (Theorem \ref{3.19} (b))). Conversely, there exist some extensions $R\subset S\subset T$ (see  \cite[Theorem 4.1 (viii) and (xii)]{DS}) such that $R\subset S$ and $ S\subset T$ are minimal, but such that $\ell[R,T]>2$.  And worse,  $\ell[R,T]=\infty $ in \cite[Remark 2.9(c)]{DPP2}. The following table shows, for each case of Theorem \ref{3.19}, which cases of \cite[Theorem 4.1]{DS} may occur.
 \vskip 0,5cm
 
 \centerline{\begin{tabular} {|c|c|}
 \hline
 Theorem  \ref{3.19} & \cite[Theorem 4.1]{DS} \\
 \hline
 (1) & (i), (ii), (iii), (iv) \\
 (2) & (i) \\
 (3) & (ii) \\
 (4) & (vi), (xi) \\
 (5) & (ix), (x) \\
 (6) & (xiii) \\
 (7) & (vii) \\
 (8) & (v) \\
 \hline
 \end{tabular}}
  \vskip 0,5cm
 
In \cite{D1}, D. Dobbs has characterized extensions $R\subset S$ of length 2 satisfying $|[R,S]|=3$. His results coincide with ours, but 
(Theorem~\ref{3.19}) holds for any extension of length 2, whatever the value of $|[R,S]|$.   In \cite[Remark 2.11]{D1}, he addresses the open problem to know when $|[R,S]|=3$ for an $M$-crucial extension $R\subset S$ which is either t-closed or subintegral. (Theorem~\ref{3.19}) answer  this question in (8) (a), (b) and (d) with $t=1$, for the t-closed case and in (6) for the subintegral case. In both cases, $R\subset S$ is simple.

The following Theorem generalizes the Primitive Element Theorem for ring extensions of length 2.

\begin{theorem} \label{3.21}  A simple ring extension of length 2 has FIP.
\end{theorem}

\begin{proof} Use (Theorem~\ref{3.19}).
\end{proof} 

 \begin{remark}\label{3.22}   We  emphasize   that when $R\subset S$ is co-pointwise minimal and either $R\subset S$ is a subintegral $M$-crucial extension such that $|R/M|=\infty$, or  $R/M\subset S/M$ is  radicial,   then $R\subset S$ has not FIP  and is not simple. See Theorem ~\ref{3.19} (6) and (8) (a). See also \cite[Example 6.2 (2)]{Pic 7}. 
 \end{remark}

  The next proposition shows that there is some rigidity in extensions of length 2.
   
 \begin{proposition} \label{3.231} Let $R\subset S$ be an integral   extension of length 2 (resp.;  an FMC extension). If $R$ is  integrally closed (in $\mathrm{Tot}(R)$) and not a field, $S$ is not an integral domain.
\end{proposition}

\begin{proof} It is enough to assume that $R\subset S$ has FMC. Assume that $S$ is an integral domain and let $T\in]R,S[$ so that $R\subset T$ is  minimal. Let $M:=(R:T)\in\mathrm{Max}(R)$. Then  $R_M\subset T_M$ is a minimal finite simple extension which is torsion-free over $R_M$ and with conductor $MR_M$. Seydi Lemma states that  a simple and finite extension  which is torsion-free over an integrally closed domain is free \cite[Corollaire 1.2]{HS}.  Since $R_M$ is integrally closed, $T_M$ is free over $R_M$, so that $R_M\subset T_M$ is flat. But \cite[Lemme 4.3.1]{FO} shows that $R_M$ is field, which is absurd.
\end{proof}

\end{document}